\documentclass[british, 12pt, reqno]{amsart}
\usepackage[T1]{fontenc}
\usepackage[latin9]{inputenc}
\usepackage{amsmath, amsthm, amssymb, stmaryrd}
\usepackage{geometry}
\usepackage{amsmath}
\usepackage{amsthm}
\usepackage{amssymb}
\usepackage{xcolor}

\usepackage{enumerate}

\usepackage
{hyperref}

\usepackage[colorinlistoftodos]{todonotes}

\makeatletter
\numberwithin{equation}{section}
\numberwithin{figure}{section}
\theoremstyle{plain}
\newtheorem{thm}{\protect\theoremname}[section]
\theoremstyle{plain}

\ifx\proof\undefined
\newenvironment{proof}[1][\protect\proofname]{\par
	\normalfont\topsep6\p@\@plus6\p@\relax
	\trivlist
	\itemindent\parindent
	\item[\hskip\labelsep\scshape #1]\ignorespaces
}{%
	\endtrivlist\@endpefalse
}
\providecommand{\proofname}{Proof}
\fi
\theoremstyle{remark}
\newtheorem{rem}[thm]{\protect\remarkname}
\theoremstyle{plain}
\newtheorem{lem}[thm]{\protect\lemmaname}

\makeatother

\usepackage{babel}
\providecommand{\corollaryname}{Corollary}
\providecommand{\lemmaname}{Lemma}
\providecommand{\remarkname}{Remark}
\providecommand{\theoremname}{Theorem}

\numberwithin{equation}{section}
\numberwithin{figure}{section}
\theoremstyle{plain}
\theoremstyle{plain}
\newtheorem{remark}[thm]{Remark}
\newtheorem{example}[thm]{Example}

\newtheorem{conj}[thm]{\protect Conjecture}

\def\1{{\textcolor{red} {1}}}
\def\d1{{\textcolor{red} {d-1}}}

\def \mmod {{\: \mathrm{mod}\:}}

\def \bN {\mathbb N}

\def \bR {\mathbb R}

\def \bZ {\mathbb Z}

\def \ba {\mathbf a}
\def \bb {\mathbf b}

\def \bd {\mathbf d}

\def \bi {\mathbf i}

\def \bm {\mathbf m}
\def \bn {\mathbf n}

\def \br {\mathbf r}

\def \bt {\mathbf t}

\def \bx {\mathbf x}

\def \by {\mathbf y}
\def \bz {\mathbf z}

\def \bzero {\mathbf 0}

\def \balp {{\boldsymbol{\alp}}}

\def \bgam {{\boldsymbol{\gam}}}

\def \bxi {{\boldsymbol{\xi}}}

\def \fb {\mathfrak b}

\def \fr {\mathfrak r}

\def \fB {\mathfrak B}
\def \fC {\mathfrak C}

\def \fN {\mathfrak N}

\def \fT {\mathfrak T}

\def \cA {\mathcal A}
\def \cB {\mathcal B}
\def \cC {\mathcal C}

\def \cE {\mathcal E}

\def \cH {\mathcal H}

\def \cL {\mathcal L}
\def \cM {\mathcal M}

\def \cP {\mathcal P}

\def \cR {\mathcal R}

\def \cT {\mathcal T}
\def \cU {\mathcal U}

\def \cX {\mathcal X}

\def \cZ {\mathcal Z}

\def \le {\leqslant}
\def \leq {\leqslant}
\def \ge {\geqslant}
\def \geq {\geqslant}

\def \det {\mathrm{det}}

\def \dag {\dagger}
\def \diam {\diamond}

\def \d {{\mathrm{d}}}

\def \ds1 {\mathds{1}}

\def \alp {{\alpha}}
\def \bet {{\beta}}
\def \gam {{\gamma}}
\def \del {{\delta}}
\def \eps {{\varepsilon}}

\def \lam {{\lambda}}
\def \sig {{\sigma}}
\def \ome {{\omega}}

\begin{document}

\author{Sam Chow \and Han Yu}

\address{Mathematics Institute, Zeeman Building, University of Warwick, Coventry CV4 7AL, United Kingdom}
\email{sam.chow@warwick.ac.uk}

\address{Mathematics Institute, Zeeman Building, University of Warwick, Coventry CV4 7AL, United Kingdom}
\email{Han.Yu.2@warwick.ac.uk}

\title[Multiplicative approximation on hypersurfaces]{Moment transference principles and multiplicative diophantine approximation on hypersurfaces}
\subjclass[2020]{11J83 (primary); 42A16 (secondary)}
\keywords{Littlewood's conjecture, diophantine approximation, manifolds, Fourier series}

\makeatletter
\providecommand\@dotsep{5}
\makeatother

\begin{abstract}
We determine the generic multiplicative approximation rate on a hypersurface. There are four regimes, according to convergence or divergence and curved or flat, and we address all of them. Using geometry and arithmetic in Fourier space, we develop a general framework of moment transference principles, which convert Lebesgue data into data for some other measure.
\end{abstract}

\maketitle


\section{Introduction}

A famous conjecture of Littlewood, from around 1930, asserts that if 
$x,y$ are real and 
$\eps > 0$ then there exists $n \in \bN$ such that
\[
n \| n x \| \cdot \| n y \|
< \eps,
\]
where $\| \cdot \|$ denotes the distance to the nearest integer. Einsiedler, Katok and Lindenstrauss \cite{EKL2006} showed that the set of exceptions $(x,y)$ has Hausdorff dimension zero. Nonetheless, the problem remains open, even in natural special cases such as 
$(x,y) = (\sqrt 2, \sqrt 3)$.
To our knowledge, the only non-trivial specific cases to have been resolved are when $x$ and $y$ lie in the same cubic field \cite{CSD, Peck}.

Turning to the metric theory, Gallagher \cite{Gal1962} determined the multiplicative approximation rate of a generic pair $(x,y) \in \bR^2$. We write $\lam_k$ for $k$-dimensional Lebesgue measure. For $\psi: \bN \to [0,1)$, we denote by $W_k^\times(\psi)$ the set of vectors $(x_1, \ldots, x_k) \in [0,1)^k$ such that
\[
\| n x_1 \| \cdots \| n x_k \| < \psi(n)
\]
has infinitely many solutions $n \in \bN$. 

\begin{thm} [Gallagher, 1962]
\label{GallagherThm}
Let $k \in \bN$, and let $\psi: \bN \to [0,1)$ be non-increasing. Then
\[
\lam_k(W_k^\times(\psi))
= \begin{cases}
0, &\text{if } \displaystyle
\sum_{n=1}^\infty 
\psi(n) (\log n)^{k-1} 
< \infty \\ \\
1, &\text{if } \displaystyle
\sum_{n=1}^\infty 
\psi(n) (\log n)^{k-1} 
= \infty.
\end{cases}
\]
\end{thm}

\subsection{Diophantine approximation on manifolds}

For $\psi: \bN \to [0,1)$, we denote by $W_k(\psi)$ the set of  
$(x_1, \ldots, x_k) \in [0,1)^k$ 
such that
\[
\max \{
\| n x_1 \|, \ldots, \| n x_k \| 
\} < \psi(n)
\]
has infinitely many solutions $n \in \bN$. 
Recently, there were breakthroughs by Beresnevich--Yang \cite{BY2023} and Beresnevich--Datta \cite{BD} to complete the following result (loosely speaking, non-degeneracy entails curvature). 

\begin{thm} [Beresnevich--Yang, Beresnevich--Datta]
Let $\cM$ be a non-degenerate submanifold of $\bR^k$, and let $\psi: \bN \to [0,1)$ be non-increasing. Then almost no/every point on $\cM$ lies in $W_k(\psi)$ if the series
\[
\sum_{n=1}^\infty 
\psi(n)^k
\]
converges/diverges.
\end{thm}

Following the seminal work of Kleinbock and Margulis \cite{KM1998}, this problem was posed in greater generality by Kleinbock, Lindenstrauss and Weiss \cite[Question 10.1]{KLW2004}, and subsequently coined the `Dream Theorem' in a popular survey article \cite{BRV2016}. The result was the culmination of a long line of work by many authors \cite{Ber2012, BDV2007, BVVZ2017, BVVZ2021, Ber1977, DRV1991, VV2006}. At the heart of the phenomenon lies the related, natural problem of counting rational points near manifolds \cite{BZ2010, Cho2017, Gaf2014, Hua2015, Hua2019, Hua2020, SST, SY2022, ST}. In the complementary setting of affine subspaces, Huang~\cite{Hua} established similar results subject to a generic diophantine condition on the parametrising matrix, advancing on the progress made by Kleinbock \cite{Kle2003} and by Huang and Liu \cite{HL2021}. 

\subsection{Multiplicative approximation on manifolds}

Kleinbock and Margulis \cite{KM1998} famously studied multiplicative approximation on manifolds at the level of the exponent. A point $\bx \in \bR^k$ is \emph{very well multiplicatively approximable} (VWMA) if there exists 
$\eps > 0$ such that
\[
\| n x_1 \| \cdots
\| n x_k \| < n^{-1-\eps}
\]
has infinitely many solutions 
$n \in \bN$.

\begin{thm} [Kleinbock--Margulis, 1998] Let $\cM$ be a non-degenerate submanifold of $\bR^k$. Then almost no point on $\cM$ is VWMA.
\end{thm}

A more precise dichotomy has largely eluded us.

\begin{conj} 
[Multiplicative Dream Theorem]
\label{MultDream}
Let $\cM$ be a non-degenerate submanifold of $\bR^k$, and let $\psi: \bN \to [0,1)$ be non-increasing. Then almost no/every point on $\cM$ lies in $W_k^\times(\psi)$ if the series
\[
\sum_{n=1}^\infty 
\psi(n) (\log n)^{k-1}
\]
converges/diverges.
\end{conj}

Badziahin and Levesley \cite{BL2007} solved the convergence theory for non-degenerate planar curves.

\begin{thm} [Badziahin--Levesley, 2007]
Let $I$ be an open interval, let $c_1 > c_2 > 0$. Let $f \in C^3(I)$, and assume that for almost every $x \in I$:
\[
c_1 > f'(x) > c_2,
\qquad f''(x) \ne 0.
\]
Let $\psi: \bN \to [0,1)$ with
\[
\sum_{n=1}^\infty 
\psi(n) \log n < \infty.
\]
Then
\[
\lam_1 \{ x \in I: 
(x, f(x)) \in W_2^\times(\psi) \} 
= 0.
\]
\end{thm}

A few years ago, Beresnevich, Haynes and Velani \cite{BHV2020} established this type of result on the divergence side for vertical planar lines (these are degenerate). Recall that the \emph{diophantine exponent} of $\alp \in \bR$ is
\[
\ome(\alp) =
\sup \{ w: \exists^\infty n \in \bN \quad
\| n \alp \| < n^{-w} \}.
\]

\begin{thm} [Beresnevich--Haynes--Velani, 2016+/2020]
Let $x_1 \in \bR$ with $\ome(x_1) < 3$. Let $\psi: \bN \to [0, 1)$ be a non-increasing with
\[
\displaystyle
\sum_{n=1}^\infty 
\psi(n) \log n = \infty,
\]
and denote by $W_{2,x_1}^\times(\psi)$ the set of $x_2 \in [0,1)$ such that
\[
\| n x_1 \| \cdot \| nx_2 \| < \psi(n)
\]
has infinitely many solutions $n \in \bN$. Then
\[
\lam_1(W_{2,x_1}^\times(\psi))
= 1.
\]
\end{thm}

Their results were generalised to higher dimensions and the inhomogeneous setting in \cite{C2018, CT2020, CT2024} using the structural theory of Bohr sets, and in \cite{Yu2023} using the discrepancy theory of irrational rotations, and a counting version was demonstrated in \cite{CT2023}. General planar lines were investigated in \cite{CY2024}, and we state the two most relevant findings below.

\begin{thm}
[Chow--Yang, 2019+/2024]
Let $\cL$ be a line in the plane. Then, for almost every $(\alp,\bet) \in \cL$ with respect to the induced Lebesgue measure,
\[
\liminf_{n \to \infty}
n (\log n)^2
\| n \alp \| \cdot 
\| n \bet \| = 0.
\]
\end{thm}

The \emph{simultaneous exponent} of $\bx = (x_1,\ldots,x_k) \in \bR^k$ is
\[
\ome(\bx) =
\sup \{ w: \exists^\infty n \in \bN \quad
\max \{ \| n x_1 \|,
\ldots, \| n x_k \| \}
< n^{-w} \},
\]
the \emph{dual exponent} is
\[
\ome^*(\bx) =
\sup \{ w: \exists^\infty \bn \in \bZ^k \quad
\| \bn \cdot \bx \|
< \| \bn \|_\infty^{-w} \},
\]
and the \emph{multiplicative exponent} is
\[
\ome^\times(\bx) =
\sup \{ w: \exists^\infty n \in \bN \quad
\| n x_1 \| \cdots \| n x_k \|
< n^{-w} \}.
\]

\begin{thm}
[Chow--Yang, 2019+/2024]
\label{PlanarLinesConv}
Let $a,b \in \bR$ with 
$\ome^*(a,b) < 5$ or $\ome^\times(a,b) < 4$, let
\[
\cL = \{ (x_1, x_2) \in \bR^2:
x_1 = a x_2 + b \},
\]
and let $\psi: \bN \to [0, 1)$ be a decreasing function such that
\[
\sum_{n=1}^\infty 
\psi(n) \log n < \infty.
\]
Then, for almost every $(\alp,\bet) \in \cL$ with respect to the induced Lebesgue measure, there exist at most finitely many $n \in \bN$ such that
\[
\| n \alp \| \cdot 
\| n \bet \| < \psi(n).
\]
\end{thm}

The diophantine condition in
Theorem \ref{PlanarLinesConv} was relaxed by Huang \cite{Hua} to $\ome(a,b) < 2$ and $a \ne 0$. For more about diophantine exponents and the inequalities between them, we refer the reader to \cite{CGGMS2020}.

\bigskip

Though the above investigations have much to say about the planar case $k=2$, the cases $k \ge 3$ have been largely out of reach. Our results, on the multiplicative theory of diophantine approximation on hypersurfaces in $\bR^k$, can be summarised as follows.

\medskip

\begin{center}
\begin{tabular} {c|c|c}
& curved & flat \\
\hline
convergence & all $k$ & all $k$ \\
\hline
divergence & $k \ge 5$ & $k \ge 9$
\end{tabular}
\end{center}

\medskip

In the curved convergence setting, our method only works for $k \ge 3$, but we can invoke the theorem of Badziahin and Levesley \cite{BL2007} when $k=2$. We note that the curved divergence problem remains open in the case $k=2$.

For $\psi: \bN \to [0,1)$ and $\by \in \bR^k$, let $W_k^\times(\psi, \by)$ be the set of $\bx \in [0,1)^k$ such that
\[
\| nx_1 - y_1 \| \cdots
\| nx_k - y_k \| < \psi(n)
\]
has infinitely many solutions $n \in \bN$. Our framework is robust enough to handle such an inhomogeneous shift.

\subsubsection{Curved hypersurfaces}

In what follows, for any smooth submanifold $\cM$ of $\mathbb{R}^k,$ a \emph{smooth measure} is a probability measure induced by a smooth function on $\cM$. For example, if $\cM$ is compact, then the natural Lebesgue probability measure on $\cM$ is itself a smooth measure. If $\cM$ is unbounded, then there is no Lebesgue probability measure. To bypass this issue, we can choose a bounded open subset $\cU\subset\cM$ and consider smooth measures supported on $\cU.$ 

\begin{thm} 
[Curved convergence theory]
\label{CurvedConvergence}
Let $\cU$ be a bounded open subset of a smooth hypersurface $\cM$ in $\bR^k$, in which the Gaussian curvature is non-zero, and let $\mu$ be a smooth measure supported on a compact subset of $\cU$. Let $\by \in \bR$, and let $\psi: \bN \to [0,1)$ with
\begin{equation} 
\label{SeriesConverges}
\sum_{n=1}^\infty 
\psi(n) (\log n)^{k-1} 
< \infty.
\end{equation}
Then $\mu$ almost no point on $\cU$ lies in $W_k^\times(\psi, \by)$.
\end{thm}

\begin{remark}
If the Gaussian curvature is non-vanishing on $\cM$ then it follows that, with respect to the induced Lebesgue measure, almost no point on $\cM$ lies in $W^\times_k(\psi, \by)$.
\end{remark}

\begin{thm} 
[Curved divergence theory]
\label{CurvedDivergence}
Let $k \ge 5$. Let $\cU \subset [0,1)^k$ be an open subset of a smooth hypersurface $\cM$ in $\bR^k$, in which the Gaussian curvature is non-zero, and let $\mu$ be a smooth measure supported on a compact subset of $\cU$. Let $\by \in \bR$, and let $\psi: \bN \to [0,1)$ be a non-increasing function such that
\begin{equation} 
\label{SeriesDiverges}
\sum_{n=1}^\infty 
\psi(n) (\log n)^{k-1} = \infty.
\end{equation}
Then $\mu$ almost every point on $\cU$ lies in $W_k^\times(\psi, \by)$.
\end{thm}

\begin{remark}
If the Gaussian curvature is non-vanishing on $\cM$ then it follows that, with respect to the induced Lebesgue measure, almost every point on $\cM$ lies in $W^\times_k(\psi, \by) \mmod \bZ^k$.
\end{remark}

\subsubsection{Affine hyperplanes}

Let $k \ge 2$ be an integer,
let $\cH$ be an affine hyperplane in $\bR^k$ with normal vector 
$(1,\alp_2, \ldots, \alp_k)$, and let $\mu$ be the induced Lebesgue probability measure on $\cH \cap [0,1)^k$.

\begin{thm} 
[Flat convergence theory]
\label{FlatConvergence}
Let $\psi: \bN \to [0,1)$ with
\eqref{SeriesConverges}. Let $\by \in \bR$, and assume that 
\begin{equation}
\label{ConvDiop}
\ome(\alp_2, \ldots, \alp_k) < 
\ell,
\end{equation}
where
\[
\ell = 1 +
\# \{ 2 \le i \le k:
\alp_i \ne 0 \}.
\]
Then
$
\mu(W_k^\times(\psi, \by)) = 0.
$
\end{thm}

\begin{remark} 
[Super-strong extremality]
\label{RemarkOnKleinbock}
For $\psi(n) = n^{-\tau}$, where $\tau > 1$, the result was obtained by Kleinbock \cite[Corollary 5.7]{Kle2003}. Kleinbock had a slightly more relaxed diophantine condition, and showed it to be necessary and sufficient for $\cH$ to have this property, known as \emph{strong extremality}. We do not believe that our diophantine condition is optimal, as an optimal criterion should take the location of $\cH$ into account. However, in the case where $\cH$ passes through the origin, our condition is optimal up to the endpoint, by Kleinbock's result.
Moreover, our threshold \eqref{SeriesConverges} is sharp and matches Gallagher's. We also have a corresponding divergence theory
(see below).
\end{remark}

\begin{thm} 
[Flat divergence theory]
\label{FlatDivergence}
Let $\psi: \bN \to [0,1)$ be non-increasing function with \eqref{SeriesDiverges}.
Let $\by \in \bR$,
and assume that 
\begin{equation}
\label{FlatDivergenceCondition}
k > 4 +
2\max \{
\ome^*(\alp_i, \alp_j): 
2 \le i < j \le k
\}.
\end{equation}
Then
$
\mu(W_k^\times(\psi, \by)) = 1.
$
\end{thm}

\begin{remark}
For Lebesgue almost all choices of $\alpha_2,\dots,\alpha_k$, we have 
\[
\omega^*(\alpha_i,\alpha_j) = 2 \qquad (i \ne j).
\]
Moreover, by the Schmidt subspace theorem \cite[Chapter VI, Corollary 1E]{Sch1980}, if $\alpha_2,\dots,\alpha_k$ are algebraic and pairwise $\mathbb{Q}$-linearly independent with 1, then $\omega^*(\alpha_i,\alpha_j)=2$ for all pairs.
In these cases, we require that $k \ge 9$. 
\end{remark}

\subsection{The method}
\label{sec: overview of method}

Let $\mu$ be a probability measure on $\bR^k$. Then
\[
W_k^\times :=
W_k^\times(\psi, \by) 
= \limsup_{n \to \infty} A_n^\times,
\]
where
\begin{equation}
\label{An}
A_n^\times = 
\{ \bx \in [0,1)^k:
\| n x_1 - y_1 \| \cdots \| n x_k - y_k \|
< \psi(n) \}.
\end{equation}
By the first Borel--Cantelli lemma, if
\[
\sum_{n=1}^\infty \mu(A_n^\times) < \infty
\]
then $\mu(W_k^\times) = 0$. By the divergence Borel--Cantelli lemma \cite{BV2023}, if
\[
\sum_{n=1}^\infty \mu(A_n^\times) = \infty
\]
and
\[
\sum_{m,n \le N} \mu(A_m^\times \cap A_n^\times)
= \left(\sum_{n \le N} 
\mu(A_n) \right)^2
+ o\left(
\left( \sum_{n \le N} \mu(A_n^\times) \right)^2
\right)
\]
then $\mu(W_k^\times) = 1$.
Our task, therefore, is to estimate the first and second moments of $\sum_{n \le N} 1_{A_n^\times}$ with respect to $\mu$. We achieve this by comparison to Lebesgue measure. In the sequel, we write $\lam$ for Lebesgue measure on $[0,1]^k$.

For Borel sets $\cA_1, \cA_2, \ldots \subseteq \bR^k$, define
\[
E_N(\mu) = \sum_{n \le N} \mu(\cA_n)
\]
and
\[
V_N(\mu) = \int \left(
\sum_{n \le N}
(1_{\cA_n}(\bx) - \lam(\cA_n))
\right)^2 \d \mu(\bx).
\]
The \emph{expectation transference principle} (ETP) holds if
\[
E_N(\mu) = (1+o(1)) E_N(\lam) + O(1)
\qquad (N \to \infty).
\]
The \emph{variance transference principle} (VTP) holds if
\[
V_N(\mu) = V_N(\lam)
+ o(E_N(\lam)^2) + O(1),
\]
along a sequence of $N \to \infty$. The broad strategy is roughly as follows.
\begin{enumerate}
\item \label{StepETPVTP}
Establish ETP and VTP for $(A_n^\times)_{n=1}^\infty$.
\item \label{StepSecondClose}
Use ETP and VTP to prove that
\begin{equation}
\label{SecondMomentClose}
\sum_{m,n \le N} \mu(A_m^\times \cap A_n^\times)
= \sum_{m,n \le N} 
\lam(A_m^\times \cap A_n^\times)
+ o(E_N(\mu)^2) + O(1),
\end{equation}
along a sequence of $N \to \infty$.
\item \label{StepLebesgueFirst}
Note that
\[
\sum_{n=1}^\infty \lam(A_n^\times)
\]
converges/diverges. 
\item \label{StepLebesgueSecond}
Show that
\begin{equation}
\label{LebesgueSecond}
\sum_{m,n \le N} 
\lam(A_m^\times \cap A_n^\times)
= (1+o(1)) E_N(\lam)^2 + O(1) \qquad (N \to \infty).
\end{equation}
\end{enumerate}
Step \eqref{StepSecondClose}, and the sufficiency of these steps for proving our theorems, will be implemented in the next section.

The sets $A_n^\times$ are unions of hyperbolic regions, which we decompose into hyper-rectangles (boxes). We can then write down their Fourier coefficients. To compute the first and second moments $E_N(\mu)$ and $V_N(\mu)$, we need to weight these by the Fourier coefficients of $\mu$. The zero frequency contributes $E_N(\lam)$ and $V_N(\lam)$, respectively, so we need to bound the contributions of the other frequencies. In the curved setting, the key to our success is the strong pointwise Fourier decay of the measure $\mu$. In the flat setting, the key to our success is the geometric constraint imposed by the uncertainty principle. In both settings, divisibility and geometry in discrete Fourier space are essential.

We establish \eqref{LebesgueSecond} by adapting Gallagher's method \cite{Gal1962}. Applying our framework to $\mu = \lam$ gives an alternate proof of the fully-inhomogeneous Gallagher theorem, which was posed explicitly in two dimensions by Beresnevich, Haynes and Velani \cite[Conjecture 2.1]{BHV2020}. This was solved by Chow and Technau, and generalised to higher dimensions in \cite[Corollary 1.11]{CT2024}. When $k \ge 3$, we are able to somewhat relax the monotonicity assumption.

\begin{thm} 
\label{FIG}
Let $\by \in \bR^k$, and let $\psi: \bN \to [0,1)$. 
\begin{enumerate}[(a)]
\item Assuming \eqref{SeriesConverges}, we have $\lam_k(W^\times(\psi;\by))  = 0$.
\item Assume that $k \ge 3$ and
\begin{equation}
\label{SomewhatSmall}
\psi(n) \ll n^{-\eps},
\end{equation}
for some constant $\eps \in (0,1]$. Then, assuming \eqref{SeriesDiverges}, we have 
\[
\lam_k(W^\times(\psi;\by)) = 1.
\]
\end{enumerate}
\end{thm}

\begin{remark}
\label{refinement}
The assumption \eqref{SomewhatSmall} is weaker than the assumption of monotonicity, in the divergence setting, irrespective of the value of the constant $\eps \in (0,1]$. Indeed, suppose $\psi(n_j) > n_j^{-\eps}$ for some positive integers $n_1, n_2, \ldots$ with
\[
n_{j+1} > 2n_j \qquad (j \in \bN).
\]
Then
\[
\psi(n) > \frac1{2n} \qquad (n_j/2 < n \le n_j),
\]
and we can replace $\psi$ by the function
\[
\psi_1(n) = \min \{ \psi(n), 1/(2n) \}.
\]

An example of non-monotonic function to which the result applies is
\[
\psi(n) = \begin{cases}
\frac1{n^{1/2} (\log n)^k}, &\text{if } n \text{ is a square} \\
0,&\text{otherwise}.
\end{cases}
\]
\end{remark}

\subsection{Open problems}

\subsubsection{Precise diophantine criteria}

In the flat settings, our diophantine criteria are sufficient but not necessary. We challenge the reader to refine these.

\subsubsection{Multiplicative divergence theory for planar curves}

The curved divergence theory remains at large here. 

\subsubsection{Beyond codimension one (hypersurfaces)}

We expect that the methods of this paper, when combined with those of \cite{BdS2023}, will deliver progress in this direction. We hope to address this in future work.

\subsubsection{Multiplicative approximation on fractals}

We expect that the methods of this paper, when combined with those of \cite{CVY, DJ, Yu}, will deliver progress in this direction. We hope to address this in future work.

\subsubsection{More general measures}

We speculate that Conjecture \ref{MultDream} might extend to friendly measures --- which generalise natural measures on manifolds and fractals --- cf. \cite[\S 10]{KLW2004}.

\subsection*{Organisation}

In \S \ref{GeneralFramework}, we reduce the problems to smooth analogues of Steps \eqref{StepETPVTP}, \eqref{StepLebesgueFirst} and \eqref{StepLebesgueSecond} of the broad strategy described above. In \S \ref{sec: rectangle decomposition}, we decompose our hyperbolic regions into rectangular ones, and introduce some smoothing. In \S \ref{sec: ETP}, we establish ETPs in the curved and flat settings. Then, in \S \ref{sec: VTP}, we establish the corresponding VTPs. We estimate the Lebesgue second moment in \S \ref{sec: Lebesgue}. Finally, in \S \ref{finale}, we complete the proofs of our main theorems.

\subsection*{Notation}

For complex-valued functions $f$ and $g$, we write $f \ll g$ or $f = O(g)$ if there exists $C$ such that $|f| \le C|g|$ pointwise, we write $f \sim g$ if $f/g \to 1$ in some specified limit, and we write $f = o(g)$ if $f/g \to 0$ in some specified limit. We will work in $k$-dimensional Euclidean space, and the implied constants will always be allowed to depend on $k$, as well as on any parameters which we described as being fixed. Any further dependence will be indicated with a subscript.

For $x \in \bR$, we write $e(x) = e^{2 \pi i x}$. For $r > 0$, we write $B_r$ for the closed ball of radius $r$ centred at the origin.

\subsection*{Funding}

HY was supported by the Leverhulme Trust (ECF-2023-186).

\subsection*{Rights}

For the purpose of open access, the authors have applied a Creative Commons Attribution (CC-BY) licence to any Author Accepted Manuscript version arising from this submission.

\section{A general framework}
\label{GeneralFramework}

The purpose of this section is to reduce our theorems to functional analogues of Steps \eqref{StepETPVTP}, \eqref{StepLebesgueFirst} and \eqref{StepLebesgueSecond} from \S \ref{sec: overview of method}.

\subsection{A setwise framework}

We begin with the setwise version of the framework. In this context, we will also be able to implement Step \eqref{StepLebesgueFirst} generally.

\begin{lem} 
[Step \eqref{StepSecondClose}]
\label{SecondMomentCloseLemma}
Let $\cA_1, \cA_2, \ldots$ be Borel subsets of $\bR^k$ for which ETP and VTP hold. Then 
\[
\sum_{m,n \le N}
\mu(\cA_m \cap \cA_n)
= \sum_{m,n \le N} \lam(\cA_m \cap \cA_n) + o(E_N(\mu)^2) + O(1),
\]
along a sequence of $N \to \infty$.
\end{lem}

\begin{proof} Using ETP and VTP, we compute that
\begin{align*}
\sum_{m,n \le N} \mu(\cA_m \cap \cA_n)
&= V_N(\mu) - E_N(\lam)^2
+2 E_N(\mu) E_N(\lam) \\
&= V_N(\lam) + E_N(\lam)^2 
+ o(E_N(\mu)^2) + O(1) \\
&= \sum_{m,n \le N} 
\lam(\cA_m \cap \cA_n)
+ o(E_N(\mu)^2) + O(1),
\end{align*}
along a sequence of $N \to \infty$.
\end{proof}

\begin{lem} [Step \eqref{StepLebesgueFirst}]
\label{LebesgueFirstLemma}
Let $\psi: \bN \to [0,1)$ with \eqref{SeriesConverges}. Then
\begin{equation}
\label{LebesgueFirstConv}
\sum_{n=1}^\infty \lam(A_n^\times) < \infty.
\end{equation}
If we instead have \eqref{SeriesDiverges}, then
\begin{equation}
\label{LebesgueFirstDiv}
\sum_{n=1}^\infty \lam(A_n^\times) =
\infty.
\end{equation}
\end{lem}

\begin{proof} By the Haar property of Lebesgue measure, we may assume that $\by = \bzero$. We write
\[
\sum_{n=1}^\infty \lam(A_n^\times) = 
S_1 + S_2,
\]
where
\[
S_1 = \sum_{\psi(n) < n^{-2k}}
\lam(A_n^\times),
\qquad
S_2 = \sum_{\psi(n) \ge n^{-2k}}
\lam(A_n^\times).
\]
Then
\[
S_1 \ll \sum_n n^{-2} < \infty.
\]

Observe by a change of variables that if $n \in \bN$ then
\[
\lam(A_n^\times) = 2^k 
\lam(\cC_k(\psi(n))) = \lam(\cB_k(2^k \psi(n))),
\]
where
\[
\cC_k(\rho) = \{ \bx \in [0,1/2]^k:
0 \le x_1 \cdots x_k \le \rho \}
\]
and
\[
\cB_k(\rho) = \{ \bx \in [0,1]^k:
0 \le x_1 \cdots x_k \le \rho \}
\]
Now \cite[Lemma A.1]{CT2023} gives
\[
\lam(A_n^\times) = \begin{cases}
1, &\text{if } 
\psi(n) \ge 2^{-k}\\
2^k \psi(n) \displaystyle
\sum_{s=0}^{k-1}
\frac{(-\log (2^k\psi(n)))^s}{s!},
&\text{if } 
\psi(n) < 2^{-k}.
\end{cases}
\]

First suppose that
we have \eqref{SeriesConverges}. Then, for large values of $n$, we have 
$\psi(n) < 2^{-k}$. Hence
\begin{align*}
S_2 &\ll 1 + \sum_{n^{-2k} \le  \psi(n) < 2^{-k}}
\psi(n) \sum_{s=0}^{k-1} 
\frac{(-\log(2^k\psi(n)))^s} {s!}
\\
&\ll 1 + \sum_n \psi(n) 
(\log n)^{k-1} < \infty.
\end{align*}

Now suppose instead that
we have \eqref{SeriesDiverges}. By the fully-inhomogeneous version of Gallagher's theorem
\cite{CT2024}, we have
$\lam(W_k^\times) = 1$. Consequently, by the first Borel--Cantelli lemma, we must have \eqref{LebesgueFirstDiv}.
\end{proof}

\begin{lem} 
[Convergence theory conclusion]
Let $\psi: \bN \to [0,1)$ with
\eqref{SeriesConverges}, and suppose we have ETP for the sets $A_1^\times, A_2^\times, \ldots$ given by \eqref{An}. Then $\mu(W_k^\times) = 0$.
\end{lem}

\begin{proof} Lemma \ref{LebesgueFirstLemma} gives \eqref{LebesgueFirstConv}. By ETP, we thus have
\[
\sum_{n=1}^\infty \mu(A_n^\times) < \infty,
\]
and the first Borel--Cantelli lemma completes the proof.
\end{proof}

\begin{lem} 
[Divergence theory conclusion]
Let $\psi: \bN \to [0,1)$ with
\eqref{SeriesDiverges}, suppose we have ETP and VTP for the sets $A_1^\times, A_2^\times, \ldots$ given by \eqref{An}, and that we have \eqref{LebesgueSecond}. Then $\mu(W_k^\times) = 1$.
\end{lem}

\begin{proof} Lemma \ref{LebesgueFirstLemma} gives \eqref{LebesgueFirstDiv}. Combining Lemma 
\ref{SecondMomentCloseLemma} with \eqref{LebesgueSecond} and the ETP, we also have 
\[
\sum_{m,n \le N} \mu(A_m^\times \cap A_n^\times)
\sim E_N(\mu)^2,
\]
along a sequence of $N \to \infty$.
The divergence Borel--Cantelli lemma~\cite{BV2023} completes the proof.
\end{proof}

\subsection{A functional analogue}
\label{sec: functional}

Working with smooth weights will enable us to carry out Step 1. This requires us to work with functions instead of sets. The purpose of the present subsection is to formalise this.

For any measure $\mu$ on $\bR^k$ and any $f \in L^1(\mu)$, we write
\[
\mu(f) = \int_{[0,1]^k} f \d \mu.
\]
For $n \in \bN$, let $f_n: \bR^k \to [0,\infty)$ be compactly supported and measurable. Define
\[
E_N(\mu) = \sum_{n \le N} \mu(f_n)
\]
and
\[
V_N(\mu) = \int
\left( \sum_{n \le N} f_n(\bx) - \lam(f_n) \right)^2
\d \mu(\bx).
\]
The \emph{expectation transference principle} (ETP) holds if
\[
E_N(\mu) = (1+o(1)) E_N(\lam) + O(1)
\qquad (N \to \infty).
\]
The \emph{variance transference principle} (VTP) holds if
\[
V_N(\mu) = V_N(\lam) + o(E_N(\mu)^2) + O(1),
\]
along a sequence of $N \to \infty$.

Let $\cA_1, \cA_2, \ldots$ be Borel subsets of $\bR^k$. Our four steps are now as follows.

\begin{enumerate}[(1)]
\item \label{FunctionalStepETPVTP}
Establish ETP and VTP for $(f_n)_{n=1}^\infty$.
\item \label{FunctionalStepSecondClose}
Use ETP and VTP to prove that
\[
\sum_{m,n \le N} \mu(f_m f_n) = \sum_{m,n \le N} \lam(f_m f_n) + o(E_N(\mu)^2) + O(1),
\]
along a sequence of $N \to \infty$.
\item \label{FunctionalStepLebesgueFirst}
Note that
\[
\sum_{n=1}^\infty \lam(f_n)
\]
converges/diverges.
\item 
\label{FunctionalStepLebesgueSecond}
Show that
\begin{equation} \label{FunctionalLebesgueSecond}
\sum_{m, n \le N} 
\lam(f_m f_n) \le C E_N(\lam)^2 + O(E_N(\lam)).
\end{equation}
Here $C > 1$ is a constant that can be taken arbitrarily close to 1 by adjusting the smooth weights.
\end{enumerate}

\begin{lem} 
[Functional Step \eqref{StepSecondClose}]
\label{FunctionalSecondMomentCloseLemma}
Suppose ETP and VTP hold for the sequence $(f_n)_{n=1}^\infty$. Then 
\[
\sum_{m,n \le N}
\mu(f_m f_n)
= \sum_{m,n \le N} 
\lam(f_m f_n) + o(E_N(\mu)^2) + O(1),
\]
along a sequence of $N \to \infty$.
\end{lem}

\begin{proof} Follow the proof of Lemma \ref{SecondMomentCloseLemma}, \emph{mutatis mutandis}.
\end{proof}

\begin{lem}
[Functional convergence theory conclusion]
\label{FunctionalConvergence}
Let $\psi: \bN \to [0,1)$. Assume that $f_n \gg 1$ on $A_n^\times$ for all sufficiently large $n \in \bN$, that
\begin{equation}
\label{ConvTech}
\sum_{n=1}^\infty \lam(f_n) < \infty,
\end{equation}
and that we have ETP for this sequence of functions. Then $\mu(W_k^\times(\psi)) = 0$.
\end{lem}

\begin{proof}
By ETP, we have
\[
\sum_{n=1}^\infty \mu(f_n) < \infty.
\]
Now
\[
\sum_{n=1}^\infty \mu(A_n^\times) \ll
1 + \sum_{n=1}^\infty \mu(f_n) 
< \infty,
\]
and the first Borel--Cantelli lemma completes the proof.
\end{proof}

We require a functional analogue of the divergence Borel--Cantelli lemma.

\begin{lem} [Functional divergence Borel--Cantelli]
\label{FunctionalDBC}
Let $\mu$ be a probability measure and let $C \ge 1$. Let $\cE_1, \cE_2, \ldots$ be measurable sets, and put 
\[
\cE_\infty = \displaystyle \limsup_{n \to \infty} \cE_n.
\]
For $n \in \bN$, let $f_n: \bR^k \to [0,\infty)$ be a measurable function supported on $\cE_n$. Assume that $\displaystyle \sum_{n=1}^\infty \mu(f_n) = \infty$, and that
\[
\sum_{m,n \le N} \mu(f_m f_n) \le C E_N(\mu)^2
\]
holds for infinitely many $N \in \bN$. Then $\mu(\cE_\infty) \ge 1/C$.
\end{lem}

\begin{proof} Denote
\[
f = \sum_{X < n \le Y} f_n,
\qquad
\cE = \bigcup_{X < n \le Y} \cE_n.
\]
By Cauchy--Schwarz,
\[
\mu(f)^2
= \mu(1_\cE f)^2
\le \mu(\cE) \mu(f^2).
\]
As
\[
\mu(f) = \sum_{X < n \le Y} \mu(f_n)
\]
and
\[
\mu(f^2) = \sum_{X < m,n \le Y} \mu(f_m f_n),
\]
we thus have
\begin{align*}
\mu(\cE) \ge \frac1C \left(
\frac{\displaystyle \sum_{X < n \le Y} \mu(f_n)}{\displaystyle \sum_{n \le Y} \mu(f_n)} \right)^2
\end{align*}
for infinitely many $Y \in \bN$. Taking $Y \to \infty$ yields
\[
\mu \left( \bigcup_{n > X} \cE_n \right) \ge \frac1C.
\]
Since
\[
\cE_\infty = \bigcap_{X=1}^\infty \bigcup_{n > X} \cE_n,
\]
we finally obtain
\[
\mu(\cE_\infty) = \lim_{X \to \infty} \mu \left( \bigcup_{n > X} \cE_n \right) \ge \frac1C.
\]
\end{proof}

\begin{lem} 
[Functional divergence theory conclusion]
\label{FunctionalDivergence}
Let $\psi: \bN \to [0,1)$. For $n \in \bN$, suppose $f_n: \bR^k \to [0,\infty)$ is supported on $A_n^\times$. Suppose we have ETP and VTP for this sequence of functions, as well as 
\begin{equation}
\label{DivTech}
\sum_{n=1}^\infty \lam(f_n) = \infty
\end{equation}
and the estimate \eqref{FunctionalLebesgueSecond}. Then $\mu(W_k^\times(\psi)) \ge 1/C$.
\end{lem}

\begin{proof} 
Using the divergence of $\sum_n \lam(f_n)$, the ETP gives
\[
\sum_{n=1}^\infty \mu(f_n) = \infty.
\]
Combining Lemma 
\ref{FunctionalSecondMomentCloseLemma} with \eqref{FunctionalLebesgueSecond} and the ETP, if $C' > C$ then we also have 
\[
\sum_{m,n \le N} \mu(f_m f_n)
\le C' E_N(\mu)^2,
\]
along a sequence of $N \to \infty$. Lemma \ref{FunctionalDBC} completes the proof.
\end{proof}

\section{Rectangular decomposition}
\label{sec: rectangle decomposition}

Let $n \in \bN$. Let us begin with one collection of scales 
\[
d_1, d_2, \ldots, d_k \in (0,1/(2n)].
\]
We associate to each $d_j$ an interval $\cU_j$ which is $[0, nd_j)$ or $[nd_j/2, nd_j)$, and define
\[
A_n(\bd) = \{ \bx \in [0,1]^k: \| n x_j - y_j \| \in \cU_j \quad \forall j \}.
\]
Next, we smooth in each dimension. Let $b_1, b_2, \ldots, b_k: \bR \to [0,1]$ be non-zero bump functions supported on $[-2,2]$. Then
\begin{equation}
\label{decay}
\hat b_j(\xi) \ll_L 
(1 + |\xi|)^{-L} 
\qquad (1 \le j \le k),
\end{equation}
for any $L \ge 1$. 

\begin{remark} These will approximate $\{x: |x| \le 1 \}$ or $\{ x: 1/2 \le |x| \le 1 \}$. In the convergence setting, we only require the former.
\end{remark}

Define
\begin{align*}
\cR &= [-d_1, d_1] \times
\cdots \times [-d_k, d_k], \\
\cR^\vee &= [-1/d_1,1/d_1] \times \cdots \times [-1/d_k, 1/d_k], \\
\bb_\cR (\bx) &= \prod_{j \le k} b_j(x_j/d_j) \qquad 
(\bx \in \bR^k).
\end{align*}
Then
\[
\hat \bb_\cR(\bxi) = \prod_{i \le k} d_j \hat b_j(d_j \xi_j)
\qquad (\bxi \in \bR^k).
\]
By \eqref{decay}, we thus have
\begin{equation}
\label{decay2}
\frac{\hat \bb_\cR(\bxi)}
{d_1 \cdots d_k} \ll_L 2^{-mL}
\qquad
(\bxi \in 2^m \cR^\vee \setminus 2^{m-1} \cR^\vee).
\end{equation}
These considerations can be used to establish ETP and VTP for unions of rectangles, via the smooth approximation
\[
A_n^\dag(\bd; \bx) = 
\sum_{\ba \in \{1,2,\ldots,n\}^k} \bb_\cR \left( \bx - \frac{\ba + \by} n \right)
\]
to the indicator function of $A_n(\bd)$. For Fourier-analytic purposes, it helps to extend this periodically to
\begin{equation}
\label{AnStar}
A_n^*(\bd; \bx)
= \sum_{\bm \in \bZ^k}
A_n^\dag(\bd; \bx + \bm) = 
\sum_{\ba \in \bZ^k} \bb_\cR \left( \bx - \frac{\ba + \by}n \right).
\end{equation}

\bigskip

To see how all of this applies to the `hyperbolic' sets $A_n^\times$, define
\[
h_i = 2^{-i}
\]
for each integer $i$ in some range. The range is
\[
\frac{\psi(n)}{n} \le 
h_i \le \frac1n
\]
in the convergence setting. In the divergence setting, for $n \in [2^{m-1}, 2^m)$, the range is
\[
\frac{1}{2^{(1+(1+\tau)/k)m}} \ll
h_i \ll \frac{1}{2^{(1+(1-\tau)/k)m}},
\]
for some small constant $\tau > 0$. Define
\begin{align*}
\fB_n &= \bigsqcup_{i_1,\ldots,
i_{k-1}}
A_n(h_{i_1}, \ldots, 
h_{i_{k-1}}, h), \\
\fC_n &=
\bigcup_{i_1,\ldots,
i_{k-1}}
A_n(h_{i_1}, \ldots, h_{i_{k-1}}, 2^{k-1}H),
\end{align*}
where
\[
h h_{i_1} \cdots h_{i_{k-1}}
= \frac{\psi(2^m)}{2^{km}},
\qquad
H h_{i_1} \cdots h_{i_{k-1}}
= \frac{\psi(n)}{n^k}.
\]

\begin{remark} For ease of notation, we have suppressed the weights $\bb$, which depend on $\bi$, and which can be different for $\fB_n$ compared to $\fC_n$.
\end{remark}

We then have
\[
\fB_n \subseteq A_n^\times \subseteq \fC_n.
\]
We will discuss this setup further in Example \ref{MultOK}.

Note that there are
$\asymp \log^{k-1}(1/\psi(n))$
many choices of $(i_1,\ldots,i_{k-1})$ in the convergence setting, whereas our results involve the quantity $(\log n)^{k-1}$. Furthermore, in the divergence setting where we will use $\fB_n$, the value of the product $h h_{i_1} \cdots h_{i_{k-1}}$ involves $\psi(2^m)$ instead of $\psi(n)$, and there is nothing to suggest that these two quantities are at all close. We will settle these differences in \S \ref{finale}.

\bigskip

In order to be able to estimate the Lebesgue second moment, we also require $\fB_n$ to have Property $\cP_n(\by)$, which we now define similarly to Gallagher's Property $\cP$ from \cite{Gal1962}. A set $\cZ \subseteq [-1/2,1/2]^k$ is \emph{strongly star-shaped} if
\[
\bx \in \cZ \quad
\text{and} \quad |z_j| \le |x_j| \: \forall j \quad \Longrightarrow \quad \bz \in \cZ.
\]
A set $\cX$ in $\bR^k$ has Property $\cP_n(\by)$ if 
\[
\cX = \cZ + n^{-1} (\by + \bZ^k),
\]
for some strongly star-shaped set $\cZ$. 

\section{Expectation transference principle}
\label{sec: ETP}

\subsection{Curved ETP}

This subsection is devoted to establishing the following theorem.

\begin{thm}[Pointwise curved ETP]
\label{CurvedETP}
Let $n \in \bN$ and
$A_n^*(\bx) = 
A_n^*(\bd; \bx)$. Let $\mu$ be a probability measure on $[0,1]^k$ such that 
$\hat \mu(\bxi) \ll (1 + |\bxi|)^{-\sig}$, for some fixed $\sig \in (0,k)$. Then
\[
\mu(A_n^*) = \lam(A_n^*)
\left(1 + O\left(
\frac{(d_1 \cdots d_k)^{-(1-\sig/k)}}{n^k} \right)
\right).
\]
\end{thm}

As $b_1, b_2, \ldots, b_k$ are Schwartz, we have an absolutely and uniformly convergent Fourier series expansion
\[
A_n^*(\bx) = \sum_{\bxi \in \bZ^k} \hat A_n^*(\bxi) 
e(\langle \bxi, \bx \rangle)
\]
where, for $\bxi \in \bZ^k$,
\begin{align}
\notag
\hat A_n^*(\bxi) &= \int_{[0,1]^k}
A_n^*(\bx) e(- \langle \bxi, \bx \rangle) \d \bx \\
\notag
&= \int_{[0,1]^k} \sum_{\bm \in \bZ^k} A_n^\dag(\bx + \bm) 
e(- \langle \bxi, \bx \rangle) \d \bx \\
\label{agreement}
&= \int_{[0,1]^k} \sum_{\bm \in \bZ^k} A_n^\dag(\bx + \bm) 
e(- \langle \bxi, \bx + \bm \rangle) \d \bx = \hat A_n^\dag(\bxi).
\end{align}
This standard calculation explains that the Fourier coefficient of $A_n^*$ equals the Fourier transform of $A_n^\dag$ at integer vectors.

Since $\lam(A_n^*) = \hat A_n^*(\bzero)$, we see from Parseval's theorem that
\[
\int_{[0,1]^k} 
(A_n^* - \lam(A_n^*)) \d \mu 
= \sum_{\bxi \ne \bzero} 
\hat A_n^*(\bxi) \hat \mu(-\bxi).
\]
Writing
\[
a_n(\bxi) = \frac{\hat A_n^*(\bxi)}{\lam(A_n^*)},
\]
we thus have
\begin{align*}
|\mu(A_n^*) - \lam(A_n^*)| &= 
\left| \int_{[0,1]^k} 
(A_n^* - \lam(A_n^*)) 
\d \mu \right| \\
&\le \lam(A_n^*)
\sum_{\bxi \ne \bzero} 
|a_n(\bxi) \hat \mu(-\bxi)|.
\end{align*}
We note that
\[
\lam(A_n^*) =  d_1 \cdots d_k n^k \prod_{j \le k} \| b_j \|_1.
\]

We assume, without loss of generality, that $d_1 \ge \cdots \ge d_k$, and write $S_i = d_i^{-1}$ for all $i$, so that
\[
S_1 \le S_2 \le \cdots \le S_k.
\]
As $A_n^*$ is $n^{-1}$-periodic, its normalised Fourier coefficients $a_n$ vanish away from $n \bZ^k$. Moreover, by the uncertainty principle, they are essentially supported on
\[
\cR_n^\vee = [-S_1, S_1] \times \cdots \times [-S_k, S_k].
\]
We formalise this below.

\begin{lem} \label{EssentialSupport}
Let $L, t \ge 1$ and $\bxi \in \bZ^k \setminus t \cR_n^\vee$. Then
\[
a_n(\bxi) \ll_L t^{-L}.
\]
\end{lem}

\begin{proof} Using \eqref{agreement} and then \eqref{decay2}, we compute that
\begin{align*}
\hat A_n^*(\bxi) &=
\hat A_n^\dag(\bxi) =
\sum_{\ba \in \{1,2,\ldots,n\}^k} e( - \langle \bxi, 
\ba + \by \rangle /n) \hat b_{\cR_n} (\bxi) \\ &\ll n^k d_1 \cdots d_k t^{-L}
\ll \lam(A_n^*) t^{-L}.
\end{align*}
\end{proof}

Let us now resume proving Theorem \ref{CurvedETP}. Define
\[
E_0 = 
\sum_{\bzero \ne \bxi \in \cR_n^\vee} 
|a_n(\bxi) \hat \mu(-\bxi)|
\]
and
\[
E_u =
\sum_{\bxi \in 2^u \cR_n^\vee \setminus 2^{u-1} \cR_n^\vee} 
|a_n(\bxi) \hat \mu(-\bxi)|
\qquad (u \in \bN),
\]
so that
\begin{equation}
\label{CurvedETPdecomp}
|\mu(A_n^*) - \lam(A_n^*)| \le \lam(A_n^*) \sum_{u=0}^\infty E_u.
\end{equation}
As $|a_n(\bxi)| \le 1_{n \mid \bxi}$, Lemma \ref{EssentialSupport} yields
\[
E_u \ll_L 2^{-uL}  
\sum_{\substack{\bzero \ne \bxi \in 2^u \cR_n^\vee \\ n \mid \bxi}} |\hat \mu(-\bxi)|
\qquad (u \ge 0).
\]
Theorem \ref{CurvedETP} now follows immediately from \eqref{CurvedETPdecomp} and the following estimate.

\begin{lem} We have
\[
\sum_{\substack{\bzero \ne \bxi \in 2^u \cR_n^\vee \\ n \mid \bxi}} |\hat \mu(-\bxi)| \ll \frac{(2^{ku} S_1 \cdots S_k)^{1-\sig/k}}{n^k} \qquad (u \ge 0).
\]
\end{lem}

\begin{proof}
For $v \in \bN$, define
\[
e_{u,v} = \sum_{n \mid \bxi \in 2^u \cR_n^\vee \cap B_{2^v} \setminus B_{2^{v-1}}} |\hat \mu(-\bxi)|,
\]
and note that this vanishes unless $2^v \ge n$ and $2^{v-1} < 2^u S_k$. Now
\[
\sum_{\substack{\bzero \ne \bxi \in 2^u \cR_n^\vee \\ n \mid \bxi}} |\hat \mu(-\bxi)| \le \sum_{j=0}^{k-1} E_{u,j},
\]
where
\begin{align*}
E_{u,0} &= 
\sum_{n \le 2^v < 2^{u+1} S_1}
e_{u,v}, \\
E_{u,j} &= \sum_{2^{u+1} S_j \le 2^v < 2^{u+1} S_{j+1}} e_{u,v}
\qquad (1 \le j \le k-1).
\end{align*}

First, suppose $n \le 2^v < 2^{u+1} S_1$. Then there are $O(2^{kv}/n^k)$ many summands, each of size $O(2^{-\sig v})$, whence
\[
E_{u,0} \ll \sum_{n \le 2^v < 2^{u+1} S_1}
\frac{(2^v)^{1-\sig/k}}{n^k}
\ll \frac{(2^{ku} S_1^k)^{1-\sig/k}}{n^k}
\ll \frac{(2^{ku} S_1 \cdots S_k)^{1-\sig/k}}{n^k}.
\]

Next, suppose $2^{u+1} S_j \le 2^v < 2^{u+1} S_{j+1}$ for some 
$j \in \{1,2,\ldots,k-1\}$. Then there are $O(2^{ju} S_1 \cdots S_j 2^{(k-j)v}/n^k)$ many summands, each of size 
$O(2^{-\sig v})$, so
\begin{align*}
e_{u,v} &\ll \frac{2^{ju} S_1 \cdots S_j 2^{(k-j-\sig)v}}{n^k}.
\end{align*}
As
\[
\prod_{i < j} \left( \frac{2^v}{2^{u+1} S_i} \right)^{\sig/k} \cdot \prod_{i > j} \left(\frac{2^{u+1} S_i}{2^v} \right)^{1-\sig/k} \ge 1,
\]
we obtain
\[
e_{u,v} \ll \frac{(2^{ku} S_1 \cdots S_k)^{1-\sig/k}}{n^k} \left(\frac{2^u S_j}{2^v} \right)^{\sig/k},
\]
whence
\[
E_{u,j} \ll \frac{(2^{ku} S_1 \cdots S_k)^{1-\sig/k}}{n^k}.
\]
\end{proof}

\subsection{Flat ETP}
\label{FlatETPsection}

Fix an integer $k \ge 2$, and let $\cH$ be an affine hyperplane with normal vector 
\[
\balp = (1, \alp_2, \alp_3, \ldots, \alp_k).
\]
Fix 
\[
\cB = [-B_1, B_1] \times \dots \times [-B_{k-1}, B_{k-1}] \subset \bR^{k-1},
\]
and let $r_1, \ldots, r_{k-1} > 0$ be small, positive constants.
Let $\eta > 0$ be small in terms of $\balp, \cB, \br$. Our implied constants are allowed to depend on $\balp, \cB, \br$, but not on $\eta$. We will construct smooth, compactly-supported probability measures $\mu_0$ on $\cH$. 

Let
\[
\fb_\eta^\dag(\bx) =
b_1(x_1/r_1) \cdots b_{k-1}(x_{k-1}/r_{k-1}) b_k(x_k/\eta), \qquad
\mu_\eta^\dag = 
\frac{\mu_k * \fb_\eta^\dag}{\| \mu_k * \fb_\eta^\dag \|_1},
\]
where $\mu_k$ is Lebesgue measure on the hyperplane 
$\{ x_k = 0 \}$ intersected with $\cB$. Then $\mu_\eta^\dag$ is a Schwartz function that is essentially supported on
\[
\cR^\dag = \cB \times [-\eta, \eta].
\]
Its Fourier transform $\hat \mu_\eta^\dag$ is essentially supported on
\[
\cR^\vee_\dag = 
\cB^\vee \times [-\eta^{-1}, \eta^{-1}],
\]
where $$\cB^\vee = [-B_1^{-1}, B_1^{-1}] \times \cdots \times [-B_{k-1}^{-1}, B_{k-1}^{-1}].$$
As $\eta \to 0$, the measure $\mu_\eta^\dag$ converges weakly to a smooth probability measure on the hyperplane $\{ x_k = 0 \}$.

We apply a rotation to $\mu_\eta^\dag$, $\cR^\dag$ and $\cR_\dag^\vee$ to obtain $\mu_\eta^\diam$, $\cR$, and $\cR^\vee$, respectively, so that the tube $\cR^\vee$ points in the direction $\balp$. Then $\mu_\eta^\diam$ is  essentially supported on $\cR$. We write $\varrho^\vee$ for this rotation in frequency space, so that $\hat \mu_\eta^\diam(\bxi)
= \hat \mu_\eta^\dag(\varrho^\vee \bxi)$
and $\varrho^\vee \cR^\vee = \cR^\vee_\dag$.
Some translate $\mu_\eta$
of $\mu_\eta^\diam$ approximates $\mu$, and 
\[
|\hat \mu_\eta (\bxi)| = |\hat \mu_\eta^\diam (\bxi)| =
\frac{|\hat \mu_k \hat \fb_\eta^\dag (\varrho^\vee \bxi)|}{\| \mu_k * \fb_\eta^\dag \|_1}
\ll \eta^{-1} |\hat \fb_\eta^\dag
(\varrho^\vee \bxi)|,
\]
so if $L,t \ge 1$ and $\bxi \in \bR^k \setminus t \cR^\vee$ then
\begin{equation} \label{HyperplaneDecay}
\hat \mu_\eta(\bxi) \ll_L t^{-L}.
\end{equation}
Finally, we define the periodic extension
\[
\mu_\eta^*(\bx) = \sum_{\bm \in \bZ^k} \mu_\eta(\bx + \bm),
\]
and note that the Fourier series of $\mu_\eta^*$ is consistent with the Fourier transform of $\mu_\eta$ in the sense that
\[
\hat \mu_\eta^* (\bxi)
= \hat \mu_\eta(\bxi) \qquad (\bxi \in \bZ^k).
\]

Note that $\mu_\eta$ converges weakly to a smooth probability measure $\mu_0$ on $\cH$, as $\eta \to 0$. Moreover, the restriction of $\mu$ to a small ball is absolutely continuous with respect to some $\mu_0$.
This subsection is devoted to establishing the following theorem. 

\begin{thm} 
[Pointwise flat ETP]
\label{FlatETP}
Let $\mu_0$ be as above, and assume that
\[
\alp_2 \cdots \alp_\ell \ne 0,
\qquad
\ome := \ome(\alp_2, \ldots, \alp_k) < \infty.
\]
Let $n \in \bN$ and
$A_n^*(\bx) = 
A_n^*(\bd; \bx)$. 
Assume that, for some fixed
$\varpi > \ome$, 
\[
\max 
\{ d_j: 1 \le j \le \ell \}
\gg \frac1{n^{1 + 1/\varpi}}.
\]
Then
\[
\mu_0(A_n^*) \sim \lam(A_n^*) 
\qquad (n \to \infty).
\]
\end{thm}

\begin{remark} We have not assumed that any of $\alp_2, \ldots, \alp_k$ are non-zero. We have written $k - \ell$ for the number of these that vanish, and re-ordered them so that the non-zero ones come first.
\end{remark}

The above result follows from the next, which has a further smoothing.

\begin{thm} 
\label{FlatSmoothETP}
Let $\mu_\eta$ be as above, and assume that
\[
\alp_2 \cdots \alp_\ell \ne 0,
\qquad
\ome := \ome(\alp_2, \ldots, \alp_k) < \infty.
\]
Let $n \in \bN$ and
$A_n^*(\bx) = 
A_n^*(\bd; \bx)$. 
Assume that, for some fixed
$\varpi > \ome$,
\[
\max 
\{ d_j: 1 \le j \le \ell \}
\gg \frac1{n^{1 + 1/\varpi}}.
\]
Then
\[
\mu_\eta^*(A_n^*) = \lam(A_n^*) (1 + O(n^{1-k})).
\]
\end{thm}

The implied constants above do not depend on $\eta$. Taking $\eta\to 0,$ we see that
\[
\mu_\eta^*(A_n^*)\to \mu_0(A^*_n).
\]
This gives Theorem \ref{FlatETP}, assuming Theorem \ref{FlatSmoothETP}.

\bigskip

\begin{proof}[Proof of Theorem \ref{FlatSmoothETP}]
As $\alp_2 \cdots \alp_\ell \ne 0$, we may assume without any loss of generality that 
\[
d_1 = \max \{ d_j: 1 \le j \le \ell \}.
\]
Much like in the previous subsection, 
\begin{align*}
\label{eqn: flat etp middle step}
\frac{\mu^*_\eta(A_n^*) - \lam(A_n^*)}{\lam(A_n^*)}
\ll_L \sum_{u=0}^\infty 2^{-uL}  
\sum_{\substack{\bzero \ne \bxi \in 2^u \cR_n^\vee \\ n \mid \bxi}} |\hat \mu_\eta(-\bxi)|.
\end{align*}
For $u, v \in \bZ_{\ge 0}$, define
\[
N^\flat(u,v)
= \# \{ \bzero \ne \bxi \in 2^u \cR_n^\vee \cap 2^v \cR^\vee: n \mid \bxi \}.
\]
It then follows from dyadic pigeonholing and \eqref{HyperplaneDecay} that
\[
\frac{\mu^*_\eta(A_n^*) - \lam(A_n^*)}{\lam(A_n^*)}
\ll_L \sum_{u,v = 0}^\infty
2^{-uL} 2^{-vL} N^\flat(u,v).
\]

Let us now consider in simple terms what $N^\flat(u,v)$ counts, namely non-zero integer vectors $\bxi = (\xi_1, \ldots, \xi_k)$ that are divisible by $n$ and subject to two geometric constraints. The first is that
\[
|\xi_j| \le \frac{2^u}{d_j}
\qquad (1 \le j \le k).
\]
We now turn to the second geometric constraint, that $\bxi \in 2^v \cR^\vee$. This is a tube of dimensions
\[
2^{v+1} \times \cdots \times 2^{v+1} \times \frac{2^{v+1}}{\eta}
\]
in the direction $\balp$, centred at the origin.  Therefore 
\[
\xi_1 \alp_j - \xi_j \ll 2^v
\qquad (2 \le j \le k).
\] 
We can write this as
\[
t_1 \alp_j - t_j \ll
2^v/n
\qquad (2 \le j \le k),
\]
where $\bxi = n \bt$. 

The number of solutions with $t_1 = 0$ is $O((2^v/n)^{k-1})$, and zero unless $2^v \gg n$. Moreover, if $L > k-1$ then
\[
\sum_{u,v = 0}^\infty
2^{-uL} 2^{-vL} (2^v/n)^{k-1} \ll_L n^{1-k}.
\]

Next, assume that $t_1 \ne 0$, write $\varpi = \ome + 2\del$, and observe that
\[
\frac{2^v}{n} \gg
\max_j \|t_1 \alp_j \| \gg
\frac1{t_1^{\ome + \del}}.
\]
We also have $|t_1| \le 2^u/(nd_1)$, so
\[
2^{u(\ome + \del)} 2^v \gg 
n(nd_1)^{\ome + \del} \gg n^{\del/\varpi}.
\]
The number of solutions with $t_1 \ne 0$ is at most a constant times
\begin{align*}
\frac{2^u}{nd_1} \left( 1 + \frac{2^v}{n} \right)^{k-1}
\ll 2^u n^{1/\varpi} 
\left( 1 + \frac{2^v}{n} \right)^{k-1}.
\end{align*}
Moreover, if 
$
L > k^3 \varpi^2(1+\del^{-1})
$
then
\[
\sum_{2^{u(\ome+\del)} 2^v \gg n^{\del/\varpi}} 
2^{-uL} 2^{-vL}
2^u n^{1/\varpi} 
\left( 1 + \frac{2^v}{n} \right)^{k-1} \ll n^{1-k}.
\]
We conclude that
\[
\frac{\mu^*_\eta(A_n^*) - \lam(A_n^*)}{\lam(A_n^*)}
\ll n^{1-k}.
\]
\end{proof}

\section{Variance transference principle}
\label{sec: VTP}

\subsection{Admissible set and function systems}
\label{admissible}

For $m \in \bN$, write
\[
D_m = [2^{m-1}, 2^m).
\]
For each $m$, we choose a finite index set $I_m$. In the analysis, we may assume this to be non-empty, for otherwise $m$ will not contribute to the relevant sums. Fix $\tau \in (0,1/k)$. For each $i \in I_m$ and each $j \in \{1,2,\ldots,k \}$, we choose $d_{i,j} = d_{i,j,m}$ such that 
\begin{align}
\label{eqn: tau}
\frac1{n^{1 + 1/k - \tau}} \gg
d_{i,j} \gg 
\frac1{n^{1 + 1/k + \tau}} \qquad (n \in D_m),
\end{align}
and such that $\prod_{j \le k} d_{i,j}$ does not depend on $i$. 
For each $n \in D_m$, we introduce a disjoint union
\[
A_n = \bigsqcup_{i \in I_m} A_n(d_{i,1},\ldots,d_{i,k}).
\]
Here we have associated an interval $\cU_{i,j}$ to $d_{i,j}$, as in \S \ref{sec: rectangle decomposition}. We also assume that $A_n$ has Property $\cP_n(\by)$, as defined in \S \ref{sec: rectangle decomposition}.
An \emph{admissible set system} is a sequence $(A_n)_{n=1}^\infty$ obtained in this way.

\begin{example} 
[Multiplicative approximation]
\label{MultOK}
We claim that in proving the divergence theory we may assume that
\[
\psi(n) \ge \psi_L(n) \qquad (n \in \bN),
\]
where
\[
\psi_L(n) = \frac1{n (\log n)^{k+1}} \qquad (n \ge 2).
\]
To see this, put $\tilde \psi = \max \{ \psi, \psi_L \}$ and note that
\[
W(\tilde \psi) = W(\psi) \cup W(\psi_L).
\]
Observe that \eqref{ConvDiop} follows from \eqref{FlatDivergenceCondition}.
As $W(\psi_L)$ has zero measure, by the convergence theory, the set $W(\tilde \psi)$ has the same measure as $W(\psi)$, and we may therefore replace $\psi$ by $\tilde \psi$. 

We may also assume that $\psi(n) \ll n^{-\eps}$, where $0 < \eps  \le 1$, by the argument given in Remark~\ref{refinement}. Since $\psi$ is monotonic, we have 
\[
\frac{\psi(n)}{n} \ge \frac{\psi(2^m)}{2^m}
\qquad (n \in D_m).
\]
We can then pack $A_n^\times$ with $\asymp m^{k-1}$ many rectangles $A_n(d_{i,1},\ldots,d_{i,k})$ as above, with
\[
\prod_{j \le k} d_{i,j} \asymp \frac{\psi(2^m)}{2^m}.
\]

We require more precision to ensure that $A_n$ has Property $\cP_n(\by)$, so let us now elaborate on the construction outlined in \S \ref{sec: rectangle decomposition}. We let $d_{i,j}$ range over integer powers of two between $d^-$ and $d^+$, where these are integer powers of 2 smaller than $2^{-1-m}$ such that
\[
d^- \asymp 2^{-m(1+(1+\tau)/k)},
\qquad
d^+ \asymp 2^{-m(1+(1-\tau)/k)}.
\]
We take
\[
d_{i,k} = \frac{\psi(2^m)}{2^{km}d_{i,1} \cdots d_{i,k-1}}
\]
and
\[
\cU_{i,j} =
\begin{cases}
[0,nd_{i,j}), &\text{if } j = k \text{ or } d_{i,j} = d^-, \\
[nd_{i,j}/2, nd_{i,j}), &\text{if } j<k \text{ and } d_{i,j} > d^-.
\end{cases}
\]

For the VTP --- which is needed for the divergence theory --- we can take $\eps = 1$. In this case 
\[
nd_{i,k} \ll n^{-1 +(1+\tau)(k-1)/k} \ll \frac1{n^{1/k-\tau}}
\]
and
\[
nd_{i,k} \gg \psi_L(n)n^{(1-\tau)(k-1)/k} \gg \frac1{n^{1/k+\tau}}.
\]
This furnishes \eqref{eqn: tau}, irrespective of the value of the constant $\tau \in (0,1/k)$.

To confirm Property $\cP_n(\by)$ for any $\eps \in (0,1]$ and any $\tau \in (0,1/k)$, suppose $\bx \in A_n$ for some $n \in D_m$. Then, for some $i \in I_m$,
\[
\| nx_j - y_j \| \in \cU_{i,j} \qquad (1 \le j \le k).
\]
Now suppose 
\[
\| n z_j - y_j \| \le
\| n x_j - y_j \| \qquad (1 \le j \le k).
\]
For $j=1,2,\ldots,k-1$, we then have
\[
\| n z_j - y_j \| \in \cU_j,
\]
where $\cU_j$ is associated to $d_j$, where $d_j$ is the least integer power of two such that
\[
d_j \ge d^-, \qquad
\| n z_j - y_j \| < d_j.
\]
The final scale is
\[
d_k = 
\frac{\psi(2^m)}{2^{km} d_1 \cdots d_{k-1}} \ge
\frac{\psi(2^m)}
{2^{km} d_{i,1} \cdots d_{i,k-1}} = d_{i,k},
\]
so indeed $\bz \in A_n(\bd) \subseteq A_n$.
\end{example}

\bigskip

We now discuss the functional analogue. For $j=1,2$, let $\ome_j: \bR \to [0,1]$ be a non-zero bump function such that $\ome_1$ is supported on $\{x:|x| < 1\}$ and $\ome_2$ is supported on $\{x: 1/2 \le |x| < 1\}$. Let $(A_n)_{n=1}^\infty$ be an admissible set system. 

Suppose $n \in D_m$. For $i \in I_m$ and $j=1,2,\ldots,k$, let 
\[
b_{i,j} = \begin{cases}
\ome_1, &\text{if } \cU_{i,j} = [0,d_{i,j}), \\
\ome_2, &\text{if }
\cU_{i,j} =
[d_{i,j}/2, d_{i,j}).
\end{cases}
\]
Furthermore, let
\[
f_n(\bx) = \sum_{i \in I_m}
A_n^*(\bd^{(i)}; \bx),
\]
where 
\[
\bd^{(i)} = (d_{i,1}, \ldots, d_{i,k}),
\]
and $A_n^*(\bd; \bx)$ is given by \eqref{AnStar}. 

An \emph{admissible function system} is a sequence $(f_n)_{n=1}^\infty$ obtained in this way.

\subsection{Curved VTP}

This subsection is devoted to establishing the following theorem.

\begin{thm}
\label{CurvedVTP}
Let $(f_n)$ be an admissible function system with sufficiently small parameter $\tau$, and let $\mu$ be a probability measure on $[0,1]^k$ such that $\hat \mu(\bxi) \ll (1 + |\bxi|)^{-2}$, where $k \ge 3$. Then VTP holds for $(f_n)_n$ and $\mu$.
\end{thm}

We assume that $\hat \mu(\bxi) \ll (1 + |\bxi|)^{-\sig}$. Later, we will put $\sigma=2.$ We begin with the Fourier series expansions
\begin{align*}
V_N(\lam) &=
\int \left(
\sum_{n \le N} (f_n(\bx) - \lam(f_n))
\right)^2 \d \lam(\bx) \\
&= \int \left(
\sum_{n \le N} \sum_{\bxi \ne \bzero} \hat f_n(\bxi) e(\bxi \cdot \bx) \right)^2 \d \lam(\bx) \\
&= \sum_{n, n' \le N} 
\sum_{\substack{\bxi, \bxi' \ne \bzero \\ \bxi + \bxi' = \bzero}}
\hat f_n(\bxi)
\hat f_{n'}(\bxi')
\end{align*}
and
\begin{align*}
V_N(\mu) &=
\int \left(
\sum_{n \le N} (f_n(\bx) - \lam(f_n))
\right)^2 \d \mu(\bx) \\
&= \int \left(
\sum_{n \le N} \sum_{\bxi \ne \bzero} \hat f_n(\bxi) e(\bxi \cdot \bx) \right)^2 \d \mu(\bx) \\
&= \sum_{n, n' \le N} \left(
\sum_{\substack{\bxi, \bxi' \ne \bzero \\ \bxi + \bxi' = \bzero}}
\hat f_n(\bxi)
\hat f_{n'}(\bxi') + \sum_{\bxi, \bxi', \bxi + \bxi' \ne \bzero} 
\hat f_n(\bxi) \hat f_{n'}(\bxi') \hat \mu(\bxi + \bxi') \right) \\
&= V_N(\lam) + E,
\end{align*}
where 
\[
E = \sum_{n, n' \le N}
\sum_{\bxi, \bxi', \bxi + \bxi' \ne \bzero} 
\hat f_n(\bxi) \hat f_{n'}(\bxi') \hat \mu(\bxi + \bxi').
\]
Note in the above manipulations that all of the Fourier series converge absolutely, since each $A_n^\dag(\bd; \cdot)$ is Schwartz and therefore so too is the Fourier transform.

To bound the error term, we dyadically pigeonhole $n$ and $n'$, 
\[
|E| \le \sum_{m,m' \le \frac{\log N}{\log 2}} \sum_{\substack{n \in D_m \\ n' \in D_{m'}}} \sum_{\bxi, \bxi', \bxi + \bxi' \ne \bzero} |\hat f_n(\bxi) \hat f_{n'}(\bxi') \hat \mu(\bxi + \bxi')|.
\]
To normalise our Fourier coefficients, we divide by
\begin{equation} \label{normalisation}
\lam(f_n) = \sum_{i \in I_m}
n^k d_{i,1} \cdots d_{i,k} \prod_{j \le k} \| b_j \|_1 \asymp 2^{km} \sum_{i \in I_m} d_{i,1} \cdots d_{i,k} =: w_m.
\end{equation}
We obtain
\[
E \ll \sum_{m,m'} w_m w_{m'} \sum_{n,n'} 
\sum_{\bxi, \bxi', \bxi + \bxi' \ne \bzero} |a_n(\bxi) a_{n'}(\bxi)\hat \mu(\bxi + \bxi')|,
\]
where
\[
a_n(\bxi) = 
\frac{\hat f_n(\bxi)}{\lam(f_n)},
\qquad
|a_n(\bxi)| \le 1.
\]
We also write
\[
a_n(\bxi) = \sum_{i \in I_m} a_{n,i}(\bxi),
\]
where
\begin{equation}
\label{ani}
\lam(f_n) a_{n,i} = 
\hat A_n^*(\bd^{(i)}(n); \bx)
\end{equation}
is the Fourier transform of $\bx \mapsto A_n^*(\bd^{(i)}(n); \bx)$.
As $\prod_{j \le k} d_{i,j}$ does not depend on $i$, we see that
\[
a_{n,i}(\bxi) \ll 
\frac{2^{km} \prod_{j \le k} d_{i,j}}{w_m} \ll
\frac1{\# I_m}.
\]

Fixing $m, m' \ll \log N$ for the time being, define
\[
S(m, m') = \sum_{\substack{n \in D_m \\ n' \in D_{m'}}}  \sum_{\bxi, \bxi', \bxi + \bxi' \ne \bzero}
|a_n(\bxi) a_{n'}(\bxi)\hat \mu(\bxi + \bxi')|.
\]
Then
\[
E \ll \sum_{m,m'} w_m w_{m'} S(m,m')
\]
and
\[
S(m,m') \le
\sum_{\substack{n \in D_m \\ n' \in D_{m'}}} \sum_{\substack{i \in I_m \\ i' \in I_{m'}}}
\sum_{\bxi, \bxi', \bxi + \bxi' \ne \bzero} |a_{n,i}(\bxi) a_{n', i'}(\bxi') \hat \mu(\bxi + \bxi')|.
\]

As $\bx \mapsto A_n^*(\bd^{(i)}; \bx)$ is $n^{-1}$-periodic, its normalised Fourier coefficients $a_{n,i}$ vanish away from $n \bZ^k$. Moreover, by the uncertainty principle, they are essentially supported on
\[
\cR_{n,i}^\vee := 
[-1/d_{i,1},
1/d_{i,1}] \times \cdots \times
[-1/d_{i,k},
1/d_{i,k}].
\]
Indeed, we have the following estimate.

\begin{lem} \label{EssentialSupport2}
Let $n \in D_m$ and $i \in I_m$. 
Let $L, t \ge 1$ and $\bxi \in \bZ^k \setminus t \cR_{n,i}^\vee$. Then
\[
a_{n,i}(\bxi) \ll_L t^{-L} (\# I_m)^{-1}.
\]
\end{lem}

\begin{proof} Write
\[
\cR_{n,i} = [-d_{i,1}, d_{i,1}] \times \cdots \times 
[-d_{i,k}, d_{i,k}].
\]
Using \eqref{ani}, then \eqref{agreement}, then \eqref{decay2}, we compute that
\begin{align*}
\lam(f_n)a_{n,i}(\bxi) &=
\sum_{\ba \in \{1,2,\ldots,n\}^k} e \left( - \left \langle \bxi, 
\frac{\ba + \by}n
\right \rangle \right ) \hat b_{\cR_{n,i}} (\bxi) \\ &\ll n^k d_{i,1}
\cdots d_{i,k} t^{-L}
\ll \frac{\lam(f_n)}{\# I_m} t^{-L}.
\end{align*}
\end{proof}

For $m, m', u, u', v \in \bN$, define
\[
N(m,m',u,u',v; i, i') = 
\sum_{\substack{n \in D_m \\ n' \in D_{m'}}} \sum_{\substack{i \in I_m \\ i' \in I_{m'}}} \sum_{\substack{0 < |\bxi + \bxi'| \le 2^v \\ n \mid \bxi \in 2^u \cR_{n,i}^\vee \setminus \{ \bzero \}
\\ n' \mid \bxi' \in 2^{u'} \cR^\vee_{n',i'} 
\setminus \{ \bzero \}
}} 1.
\]
Then
\[
S(m,m') \ll_{L,\sig} \sum_{\substack{i \in I_m \\ i' \in I_{m'}}} 
\sum_{u,u'= 1}^\infty
\sum_{v \le u + u' + 2(m+m')}
\frac{N(m,m',u,u',v; i, i')}
{2^{uL + u'L + v \sig} \#I_m
\#I_{m'}}.
\]
Write
\begin{equation}
\label{Tnx}
T = \begin{pmatrix}
t_1 & t_1' \\
t_2 & t_2' \\
\vdots & \vdots \\
t_k & t_k'
\end{pmatrix},
\qquad
\bn = \begin{pmatrix} n \\ n' \end{pmatrix},
\qquad 
\bx = 
\begin{pmatrix} x_1 \\ x_2 \\ \vdots \\ x_k
\end{pmatrix}.
\end{equation}
The quantity $N(m,m',u,u',v; i, i')$ counts integer solutions to $T \bn = \bx$ such that
\begin{align*}
&n \in D_m, \qquad
n' \in D_{m'}, \qquad
\bt, \bt', \bx \ne \bzero, \\
&|t_j n| \le 2^{u}/d_{i,j,m},
\quad
|t_j' n'| \le 2^{u'}/d_{i',j,m'}
\qquad (1 \le j \le k), \\
&|x_1|, \ldots, |x_k| \le 2^v.
\end{align*}

Put
\begin{equation}
\label{fmm'}
f(m,m') = \frac{2^{-k(m+m')}}{d_{i,1,m} \cdots d_{i,k,m} d_{i',1,m'} \cdots d_{i',k,m'}} = \frac{\# I_m \# I_{m'}}{w_m w_{m'}}.
\end{equation}
Recalling that $d_{i,j} \ll 1/n$, we find it helpful to note the trivial estimate
\[
N(m,m',u,u',v; i, i') \ll
2^{k(u+u'+v)+m+m'} f(m,m')
\]
coming solely from the constraints on the variables. 
We write 
\[
N(m,m',u,u',v; i, i') = N_1 + N_2,
\]
where $N_1$ counts solutions for which $\bt$ and $\bt'$ are parallel, and  $N_2$ counts solutions with $\bt, \bt'$ non-parallel.

\bigskip

In the case of $N_2$, the matrix $T$ contains a pair of linearly independent rows, and plainly it suffices to count solutions where the first two rows are independent. 
By symmetry, we may assume that $m \ge m'$. Since $(t_1, t_2) \ne (0,0)$, we may also assume that $t_1 \ne 0$.

We choose $t_1, t_2, t_1', x_1, x_2$. Then, since
\[
t_1 n + t_1' n' = x_1,
\qquad
t_2 n + t_2' n' = x_2,
\]
we have
\[
(t_2 t_1' - t_1 t_2')n'
= t_2 x_1 - t_1 x_2.
\]
As $(t_2 t_1' - t_1 t_2')n' \ne 0$, we must have $t_2 x_1 - t_1 x_2\neq 0$, so there are at most $O(2^{(m+m'+u+u')\eps})$ many possibilities for $n'$, by the divisor bound. Choosing $n'$ then determines at most one possibility for $n$ and $t_2'$, and choosing $t_3, \ldots, t_k, t_3', \ldots, t_k'$ then determines the remaining variables. Thus, by (\ref{eqn: tau}),
\begin{align*}
N_2 &\ll
2^{(m+m'+u+u')\eps + k(u+u') + 2v - (1-k\tau)(m+m')/(2k)} f(m,m').
\end{align*}

\bigskip

In the case of $N_1$, the vectors $\bt, \bt', \bx$ are all parallel. By symmetry, we may assume that $x_1 \ne 0$. Then $t_1 t_1' \ne 0$, since $\bt, \bt'$ are non-zero multiples of $\bx$.
Put
\begin{equation}
\label{rjdef}
r_j = \frac{2^{u+2-m}}{d_{i,j,m}} + 1,
\quad
r_j'= \frac{2^{u'+2-m'}}{d_{i',j,m'}} + 1
\qquad (1 \le j \le k).
\end{equation}
Then we may assume that
\begin{equation}
\label{WeirdSym}
2^m r_1 \le 2^{m'} r'_1.
\end{equation}

We write $t_1 = ds_1$ and $t_1' = ds_1'$ with $(s_1, s_1') = 1$. After choosing $t_1, t_1', x_1$, the congruence
\[
d s_1 n \equiv x_1 \mmod d|s_1'|
\]
determines at most one residue class for $n$ modulo $|s_1'|$, and after choosing $n$ there is at most one value of $n'$ such that $t_1 n + t_1' n' = x_1$. Next, we choose $x_2$, thereby determining 
\[
t_2 = \frac{t_1 x_2}{x_1}, \qquad
t_2' = \frac{t_2' x_2}{x_1}.
\]
Finally, for each $j \ge 3$, choosing $t_j$ or $t_j'$ determines the other as well as $x_j$. Hence, 
\begin{align*}
N_1 &\le
2^{2v} \sum_{d \le r_1} \frac{r_1}{d} \sum_{1 \le |s_1'| \le r_1'/d} \left( \frac{2^m}{|s_1'|} + 1 \right) \prod_{j=3}^k \sqrt{r_j r_j'} \\
&\ll 2^{(m+m'+u+u')\eps +2v} (2^m + r_1')r_1 \prod_{j=3}^k \sqrt{r_j r_j'} \\
&\ll 2^{(m+m'+u+u')\eps +2v + k(u+u')}
\frac{(2^m + r_1')r_1}{\sqrt{r_1 r_1' r_2 r_2'}} \sqrt {f(m,m')} \\
&\ll 2^{(m+m'+u+u')\eps +2v + (k-1/2)(u+u') - (1-k\tau)(m+m')/(2k)}
\frac{(2^m + r_1')r_1}{\sqrt{r_1 r_1'}} \sqrt {f(m,m')}.
\end{align*}

Note that
\[
\sqrt{r_1 r_1'} \ll 2^{(u+u')/2+(1+k\tau)(m+m')/(2k)} 
\]
and, by \eqref{WeirdSym},
\[
\frac{2^m r_1}{\sqrt{r_1 r_1'}} \le 2^{(m+m')/2}.
\]
Therefore
\[
N_1 \ll
2^{(m+m'+u+u')\eps +2v + k(u+u') + (m+m')/2 - (1-k\tau)(m+m')/(2k)} \sqrt{f(m,m')}.
\]
Since
$
f(m,m') \gg 2^{(1-k\tau)(m+m')},
$
we obtain
\[
N_1 \ll
2^{(m+m'+u+u')\eps +2v + k(u+u') + (O(\tau)-1)(m+m')/(2k)} f(m,m').
\]

\bigskip

Choosing $L$ large and $\sig = 2$, and combining our estimates for $N_1$ and $N_2$ gives
\begin{align*}
E &\ll \sum_{m,m' \le \frac{\log N}{\log 2}} 
\sum_{\substack{i \in I_m \\ i' \in I_{m'}}}
\sum_{u,u'=1}^\infty 
\sum_{v \le u + u' + 2(m + m')}
\frac{N(m,m',u,u',v; i, i')}{2^{uL+u'L+v\sig}
f(m,m')} \\
&\ll \sum_{m,m' \le \frac{\log N}{\log 2}}
\frac{(m+m') \# I_m \# I_{m'}}{2^{(1-O(\tau))(m+m')/(2k)}}.
\end{align*}
As
\[
E_N(\lam) \asymp \sum_{m \le \frac{\log N}{\log 2}} 2^m w_m
\]
whenever $N$ is a power of two, it now suffices to prove that
\[
\frac{m \# I_m}{2^{(1-O(\tau))m/(2k)}} = o(2^m w_m),
\]
whenever $\# I_m \ge 1$.
Note from \eqref{eqn: tau} and \eqref{normalisation}
that
\begin{equation} \label{wbound}
\frac{2^m w_m}{\#I_m} \gg 2^{-k\tau m}.
\end{equation}
Since $\tau$ is small, this completes the proof of Theorem \ref{CurvedVTP}.

\subsection{Flat VTP}

The purpose of this subsection is to establish the following theorem.

\begin{thm} [Flat VTP]
\label{FlatVTP}
Let $\cH$ be an affine hyperplane in $\bR^k$, and let $\mu_0$ be as in \S \ref{FlatETPsection}.
Write 
\[
\balp = (1,\alp_2,\ldots,
\alp_k)
\]
for the normal direction of $\cH$, and suppose
$
k > 2\ome_2(\cH) + 4,
$
where
\[
\ome_2(\cH) = \max_{2 \le i < j \le k}
\ome^*(\alp_i, \alp_j).
\]
Let $(f_n)$ be an admissible function system, where $\tau$ is small in terms of $\cH$.
Then VTP holds for $(f_n)_n$ and $\mu_0$.
\end{thm}

We will deduce Theorem \ref{FlatVTP} from the following deformation. 

\begin{thm}
\label{SmoothFlatVTP}
Let $k \ge 3$ be an integer, let $\cH$ be an affine hyperplane in $\bR^k$, let $(f_n)$ be an admissible function system, and let $\mu_\eta$ be as in \S \ref{FlatETPsection}.
Write 
\[
\balp = (1,\alp_2,\ldots,
\alp_k)
\]
for the normal direction of $\cH$, and assume that
\[
\ome_2(\cH) := \max_{2 \le i < j \le k}
\ome^*(\alp_i, \alp_j) < \infty.
\]
Then
\begin{align}
\label{eqn: FlatSmoothVTP}
|V_N(\mu_\eta) - V_N(\lam)|
\ll \sum_{n,n' \le N} \lam(f_n) \lam(f_{n'}) 
u(n, n'),
\end{align}
where
$$
u(n,n') =
(nn')^{(2\ome_2(\cH) + 4 - k + O(\tau) + o(1))/(2k)}.
$$
The implied constant above does not depend on $\eta$, and similarly the implied convergence is uniform in $\eta$.
\end{thm}

To see how Theorem \ref{FlatVTP} follows from Theorem \ref{SmoothFlatVTP}, notice that (\ref{eqn: FlatSmoothVTP}) holds uniformly in $\eta$. On the other hand, for each $N \in \bN$,
\[
V_N(\mu_\eta)\to V_N(\mu_0)
\]
as $\eta\to 0.$ Since $\tau$ is small, we deduce Theorem \ref{FlatVTP}.

\begin{proof}[Proof of Theorem \ref{SmoothFlatVTP}] We initially follow the analysis of the previous section, giving
\[
V_N(\mu_\eta) = V_N(\lam) + E,
\]
where
\[
E \ll \sum_{m,m' \le \frac{\log N}{\log 2}} w_m w_{m'} S(m,m').
\]
Here $w_m$ is as before, see \eqref{normalisation}. Similarly $S(m,m')$ is defined as before, but with a different measure $\mu_\eta$, so
\[
S(m,m') \ll \# I_m^{-1} \# I_{m'}^{-1} \sum_{u,u'=1}^\infty 2^{-uL - u'L}
\sum_{\substack{n \in D_m \\ n' \in D_{m'}}} \sum_{\substack{i \in I_m \\ i' \in I_{m'}}}
\sum_{\substack{n \mid \bxi \ne \bzero \\ n' \mid \bxi' \ne \bzero \\ \bxi + \bxi' \ne \bzero}} \hat \mu_\eta(\bxi + \bxi').
\]
The measure $\mu_\eta$ is essentially supported on the tube $\cR^\vee$ from \S \ref{FlatETPsection}, by \eqref{HyperplaneDecay}. 
We obtain
\[
S(m,m') \ll_L 
\sum_{\substack{i \in I_m \\ i' \in I_{m'}}}
\sum_{u,u',v= 1}^\infty
\frac{
N(m,m',u,u',v; i, i')}{2^{(u+u'+v)L} \#I_m \# I_{m'}},
\]
with a slight further difference as follows. Recalling \eqref{Tnx}, the quantity 
$
N(m,m',u,u',v)
$
now counts integer solutions to $T \bn = \bx$ such that
\begin{align*}
&n \in D_m, \qquad
n' \in D_{m'}, \qquad
\bt, \bt', \bx \ne \bzero, \qquad
\bx \in \cT, \\
&|t_j n| \le 2^{u}/d_{i,j},
\quad
|t_j' n'| \le 2^{u'}/d_{j,m',i'}
\qquad (1 \le j \le k).
\end{align*}
Here $C = C(\balp)$ is a positive constant, and
\[
\cT = \{ \bx \in \bR^k:
|x_j - \alp_j x_1| \le C2^v
\quad 
(2 \le j \le k) \}.
\]
We again write
\[
N(m,m',u,u',v; i, i') = N_1
+ N_2,
\]
where $N_1$ counts solutions with $\bt, \bt'$ parallel. 

\bigskip

To estimate $N_2$, we assume as we may that at least two of the first three rows of $T$ are linearly independent. 
Since $k \ge 3$, we have $t_1 n + t_1'n' = x_1$ and
\[
|t_j n + t_j' n' - \alp_j x_1| \le C2^v
\qquad (j = 2, 3).
\]
Therefore
\[
M
\begin{pmatrix}
n \\ n'
\end{pmatrix}
= \begin{pmatrix}
O(2^v) \\ O(2^v)
\end{pmatrix},
\]
where
\[
M = \begin{pmatrix}
t_2 - \alp_2 t_1 &
t_2' - \alp_2 t_1' \\
t_3 - \alp_3 t_1 & t_3' - \alp_3 t_1'
\end{pmatrix}
\]
and
\[
\det(M) =
\det
\begin{pmatrix}
t_2 &
t_2' \\
t_3 & t_3'
\end{pmatrix}
-\alp_2 \det
\begin{pmatrix}
t_1 &
t_1' \\
t_3 & t_3'
\end{pmatrix}
-\alp_3 \det
\begin{pmatrix}
t_2 &
t_2' \\
t_1 & t_1'
\end{pmatrix}
\]
As $\ome^*(\alp_2, \alp_3) < \infty$, and at least two of the first three rows of $T$ are linearly independent, we infer that $\det(M) \ne 0$. Now
$(n,n') \in M^{-1} \cB$, where $\cB$ is a ball of radius $O(2^v)$ centred at the origin.

The scaling of $\cB$ under $M^{-1}$ is characterised by its singular values 
\[
\sig_2 \ge \sig_1 > 0
\]
and we have
\[
|M^{-1} \cB \cap \bZ^2| \ll
2^{2v}(1+\sig_1)(1+\sig_2).
\]
We now estimate the sizes of $\sigma_1\sigma_2$ and $\sigma_1+\sigma_2$. Note that $\sig_1^2, \sig_2^2$ are the eigenvalues of $(M^T M)^{-1}$. With \eqref{rjdef}, we thus have
\[
\sig_1 \sig_2 = 
|\det(M)|^{-1} \ll \max_{i,j \le 3} (r_i r_j')^{\ome^*(\alp_2, \alp_3)+o(1)}.
\]
Similarly, as $\sig_1^{-2}, \sig_2^{-2}$ are the eigenvalues of $M^TM$, and the latter matrix has entries 
$O(\max \{r_1, r_2, r_3, r_1', r_2', r_3'\}^2)$, we compute that
\begin{align*}
\sig_1 + \sig_2 &= 
\sig_1 \sig_2 (\sig_1^{-1} + \sig_2^{-1}) \ll
 \sig_1 \sig_2 \sqrt{\sig_1^{-2} + \sig_2^{-2}} \\
 &\ll \max_{i,j \le 3} (r_i r_j')^{\ome^*(\alp_2, \alp_3)+o(1)}
\max \{r_1, r_2, r_3, r_1', r_2', r_3'\}.
\end{align*}
Given $t_1,t_1',t_2,t_2',t_3,t_3'$, the number of possibilities for $n$ and $n'$ is therefore bounded by a constant times
\[
2^{2v} \max_{i,j \le 3} (r_i r_j')^{\ome^*(\alp_2, \alp_3)+o(1)}
\max \{r_1, r_2, r_3, r_1', r_2', r_3'\}.
\]

After choosing $t_1,t_1',t_2,t_2',t_3,t_3'$ and $n,n'$, the variables $x_1, x_2, x_3$ are determined,
and there are $O(2^{(k-3)v})$ many possibilities for $x_4, \ldots, x_k$. Then there are at most $\min \{r_4, r_4'\}$ many possibilities for $t_4$ and $t_4'$.
Finally, there are at most $\sqrt{r_j r_j'}$ many possibilities for $t_j$ and $t_j'$, for $j= 5, 6, \ldots,k$.
Hence, using \eqref{eqn: tau},
\begin{align*}
N_2 &\ll 2^{(k-1)v} \max_{i,j} (r_i r_j')^{\ome_2(\cH) + o(1)} 
r_1 r_1' r_2 r_2' r_3 r_3' \sqrt{r_5 r_5' \cdots r_k r_k'} \\
&\qquad \cdot 
\max \{ \max_j r_j, \max_j r_j' \} \cdot 
\min \{ r_4, r_4' \} \\
&\ll 2^{O(u+u'+v)
+ (1+k\tau)(m+m')(\ome_2(\cH) + o(1))/k - (1-k\tau)(m+m')(k-2)/(2k)} f(m,m') \\
&\qquad \cdot 2^{(1+k\tau)(m+m')/k}
\\
&= 2^{O(u+u'+v)
+ (2\ome_2(\cH) + 4 - k + O(\tau) + o(1)) (m+m') / (2k)} f(m,m').
\end{align*}

\bigskip

To estimate $N_1$, we write
\[
\bt = a \bt'',
\qquad
\bt' = a' \bt'',
\]
for some primitive vector $\bt''$ and some $a, a' \in \bZ$, as well as
\[
\bx = r\bt'',
\qquad
r = an + a'n'.
\]
Note that $ana'n' \ne 0$, so $(r-an)(r-a'n') \ne 0$. After choosing $\bt''$, there are
\[
\fN(\bt'') =
\# \{ r \in \bZ \setminus \{ 0 \}: r \bt'' \in \cT \}.
\]
many choices for $r$. Next, we can either choose $a,n$ and apply the divisor bound to count possibilities for $a', n'$, or do it the other way around. Thus, for any $j$ such that $t_j'' \ne 0$, 
\[
N_1 \ll \sum_{\bt''} \fN(\bt'')
\sqrt{r_j r_j' 2^{m + m'}}
2^{o(u+u'+m+m')}.
\]
Consequently
\[
N_1 \ll \sum_{\bt''} \fN(\bt'') \fr 2^{(m+m')/2}
2^{o(u+u'+m+m')},
\]
where $\fr = \max_j \sqrt{r_j r_j'}$.

Observe that if $\bt'' \notin \cT$ then $\fN(\bt'') = 0$. If $\bt'' \in \cT$ then, upon choosing $r$ maximally, we see that
\[
\fN(\bt'')(t''_j - \alp_j t''_1)
\ll 2^v
\qquad 
(2 \le j \le k).
\]
If $t''_1 = 0$, then $\fN(\bt'') \ll 2^v$. If instead $t''_1 \ne 0$, then
\[
\fN(\bt'') \ll 2^v (t_1'')^{\ome_1(\cH) + o(1)},
\]
where
$
\ome_1(\cH) = \ome(\alp_2, \ldots, \alp_k).
$
Hence, using \eqref{eqn: tau},
\begin{align*}
N_1 &\ll 2^{v + o(u+u'+m+m') + (m+m')/2 } \fr^{1+\ome_1(\cH)} \\
&\ll 2^{O(u+u'+v+k\tau) +
(o(1)+(1+(1+\ome_1(\cH))/k)/2) (m+m')} \\
&\ll 2^{O(u+u'+v+k\tau) +
(o(1)+(-1+(1+\ome_1(\cH))/k)/2) (m+m')} f(m,m').
\end{align*}

\bigskip

Finally, using \eqref{fmm'} and the trivial inequality $\ome_2(\cH) \ge \ome_1(\cH)$, we see that if $N$ is a power of two then
\begin{align*}
E &\ll
\sum_{m, m' \le \frac{\log N}{\log 2}} w_m w_{m'} \sum_{\substack{i \in I_m \\ i' \in I_{m'}}}
\sum_{u,u',v= 1}^\infty
\frac{
N(m,m',u,u',v; i, i')}{2^{(u+u'+v)L} \#I_m \# I_{m'}}
\\ &\ll \sum_{m, m' \le \frac{\log N}{\log 2}} w_m w_{m'} f(m,m')
2^{(2\ome_2(\cH) + 4 - k + O(\tau) + o(1)) (m+m') / (2k)} \\
&= \sum_{m, m' \le \frac{\log N}{\log 2}} \# I_m \# I_{m'}
2^{(2\ome_2(\cH) + 4 - k + O(\tau) + o(1)) (m+m') / (2k)} \\
&\ll \sum_{m, m' \le \frac{\log N}{\log 2}} \frac{\# I_m \# I_{m'}}{2^{m+m'}} \sum_{\substack{n \in D_m \\ n' \in D_{m'}}}
(nn')^{(2\ome_2(\cH) + 4 - k + O(\tau) + o(1))/(2k)}.
\end{align*}
By \eqref{wbound}, we now have
\[
E \ll \sum_{n,n' \le N} \lam(f_n) \lam(f_n') u(n,n').
\]
\end{proof}

\section{The Lebesgue case}
\label{sec: Lebesgue}

In this section, we establish quasi-independence results for admissible set/function systems. These are \emph{homogeneous} if $\by = \bzero$.

\begin{thm}
\label{thm: quasi-independence part one}
Let $(A_n)_{n=1}^\infty$ be a homogeneous admissible set system. Then, for all $n,n' \in \bN$, we have
\begin{align*}
\lam(A_n\cap (A_{n'}+\bgam)) &\le \lam(A_n)\lam(A_{n'}) \\
& \quad + O \Bigg( \min 
\left \{
\frac{\lam(A_n)}{n^k}, \frac{\lam(A_{n'})}
{(n')^k} \right \} \gcd(n,n')^k \\
& \qquad \qquad \cdot \left(1+\frac{n^{\tau-1-1/k}+(n')^{\tau-1-1/k}}{\gcd(n,n')/(nn')} \right)^{k-1} \Bigg),
\end{align*}
uniformly for all $\bgam \in\mathbb{R}^k.$
\end{thm}

\begin{proof}
By the Haar property of Lebesgue measure, we may assume that $n\geq n'.$ Let us first examine the pattern of the union of the lattices $n^{-1} \mathbb{Z}^k$ and $(n')^{-1}\mathbb{Z}^k$ inside $[0,1)^k$. As $A_n,A_{n'}$ are $1/n,1/{n'}$-periodic, respectively, we look at
\[
P_{n,n'} = (n')^{-1} \mathbb{Z}^k\cap [0,1)^k\mod n^{-1}\mathbb{Z}^k.
\]
For convenience, we consider this as a subset of $[0,1/n)^k.$ If $\gcd(n,n')=1$, then there are precisely $(n')^k$ many points. They are distributed evenly inside $[0,1/n]^k$, i.e. they form the lattice $(nn')^{-1} \mathbb{Z}^k$. In general, there are still $(n')^k$ many points, when counted with their multiplicity $\gcd(n,n')^k.$ Therefore $P_{n,n'}$ is the lattice $(nn'/\gcd(n,n'))^{-1}\mathbb{Z}^k$ with each lattice point being visited $\gcd(n,n')^k$ many times. For  $\bgam \in \mathbb{R}^k$, we consider the shifted lattice
\[
P_{n,n'}(\bgam)=((n')^{-1}\mathbb{Z}^k+\bgam) \cap [0,1)^k \mod n^{-1} \mathbb{Z}^k.
\]
This again has gap $\gcd(n,n')/(nn')$ and multiplicity $\gcd(n,n')^k.$

We now consider the picture inside the torus $[0,1/n]^k \mod n^{-1} \mathbb{Z}^k.$ In this environment, the set $A_n$ is strongly star-shaped, i.e.
\[
\bx \in A_n \quad \text{and} \quad |z_j| \le |x_j| \: \forall j \quad \Longrightarrow \quad \bz \in A_n.
\]
The set $A_{n'}$ is a union of $(n')^k$ many translated copies --- possibly with exact overlaps --- of a strongly star-shaped set. We write $H_n, H_{n'}$ for the respective strongly star-shaped sets, in the torus. Let $\bt,\bt'\in [0,1/n]^k$ be vectors with $\bt\leq \bt'$ in the sense that each component of $\bt$ is smaller than the corresponding component of $\bt'$ in the one-torus $[0,1/n] \mod n^{-1} \bZ.$
Here the size of a component is its distance to the origin of the one-torus.
We infer from \cite[Lemma 1]{Gal1962} that 
\begin{equation} \label{observation}
\lam(H_n\cap (H_{n'}+\bt))
\geq \lam(H_n\cap (H_{n'}+\bt')). 
\end{equation}

\bigskip

For ease of exposition, we first assume that $\bgam = \bzero.$ Now, with $\bt_i$ as the translation vectors for copies of $H_{n'}$ in $A_{n'}$, we consider
\begin{align*}
\lambda(A_n\cap A_{n'})\leq&\sum_{i} \lam(H_n\cap (H_{n'}+\bt_i))\\
=&\sum_{\bt_i\in \partial}\lam(H_n\cap (H_{n'}+\bt_i))+\sum_{\bt_i\notin \partial}\lam(H_n\cap (H_{n'}+\bt_i)),
\end{align*}
where $\partial$ denotes the topological boundary of $[0,1/n]^{k} \subset \bR^k$.

For brevity, we refer to hypercubes simply as `cubes' in what follows. To each non-boundary $\bt_i\in (nn'/\gcd(n,n'))^{-1}\mathbb{Z}^k$, we associate a cube $C_i$ of side length $1/(nn'/\gcd(n,n'))$ such that any element of $C_i$ is at least as close to the origin as $\bt_i$, and moreover 
\[
C_i \cap C_j = \emptyset
\qquad (\bt_i \ne \bt_j).
\]
We can achieve this as follows.

\begin{itemize}
\item Decompose $[0,1/n)^k$ into $2^k$ many disjoint cubes with half the side length. Each of these only contains one torus origin.
\item If $\bt_i$ is inside one of the half-cubes, we can choose $C_i$ with one corner being $\bt_i$ and the opposite corner being closer to the torus origin, component-wise.
\end{itemize}
For example, consider $[0,1/(2n))^k.$ For each non-boundary translation vector $\bt_i\in [0,1/(2n))^k$, we choose the point 
\[
\bt_i -\frac{(1,\dots,1)}{(nn'/\gcd(n,n'))}
\]
as the opposite corner of $C_i.$

Now, by \eqref{observation}, 
\begin{align*}
&\sum_{\bt_i\notin
\partial}
\lam(H_n\cap (H_{n'}+\bt_i)) \\ &\leq \sum_{\bt_i\notin
\partial} \lam(C_i)^{-1}\int_{C_i} \lam(H_n\cap (H_{n'}+\bt)) \d \bt\\
& \leq
\gcd(n,n')^k\left(\frac{\gcd(n,n')}{nn'}\right)^{-k} \int_{[0,1/n]^k} \lam(H_n\cap (H_{n'}+ \bt)) \d \bt.
\end{align*}
Writing
\[
H_n(\bx) = \begin{cases}
1, &\text{if } \bx \in H_n \mmod n^{-1} \bZ^k \\
0, &\text{otherwise},
\end{cases}
\]
we thus have
\begin{align*}
\sum_{\bt_i\notin
\partial}\lam(H_n\cap (H_{n'}+\bt_i)) &\leq 
(nn')^k \int_{[0,1/n]^k}\int_{[0,1/n]^k} H_n(\bx) H_{n'}(\bx-\bt)\d \bx \d \bt \\ &=
(nn')^k \int_{[0,1/n]^k}\int_{[0,1/n]^k} H_n(\bx) H_{n'}(\bx') \d \bx \d \bx'\\
&= (nn')^k \lam(H_n)\lam(H_{n'}) = \lam(A_n)\lam(A_{n'}).
\end{align*}

For the boundary sum of $\bt_i,$
\[
\lam(H_n\cap (H_{n'}+\bt_i))\leq \min\{\lam(H_n),\lam(H_{n'})\}.
\]
Our task is to estimate the number of $\bt_i \in \partial$ such that $H_n$ and $H_{n'}+\bt_i$ intersect. We know that $H_n$ is contained in a cube of side length $O(n^{\tau-1-1/k})$. Using this information, we see that 
\begin{align*}
&\# \{i: \bt_i\in \partial, \quad H_n\cap (H_{n'}+\bt_i) \neq\emptyset\}
\\ &\ll 
\gcd(n,n')^k
\left(1+\frac{{n^{\tau-1-1/k} + (n')^{\tau-1-1/k}}}{\gcd(n,n')/(nn')} \right)^{k-1}.
\end{align*}
This proves the result when $\bgam=\bzero.$

\bigskip

Now let $\bgam \in \bR^k$ be arbitrary. The multiset $\fT$ of translation vectors $\bt_i$ for $H_{n'}$ is now shifted by $\bgam$ modulo $n^{-1} \bZ^k$, and the argument needs to be modified accordingly. For example, it might be that none of the $\bt_i$ are on the boundary, but some are close to it. 

Observe that $[0,1/n]^k$ can be covered by a union of cubes with disjoint interiors --- save for repetitions with multiplicity $(n,n')^k$ --- of side length $\gcd(n,n')/(nn')$ and whose corners in $[0,1/n]^k$ are in $\fT$. We call such a cube \emph{good} if it does not intersect the boundary of $[0,1/n]^k,$ and otherwise we call it \emph{bad}. We declare corners of bad cubes to be \emph{bad}, and 
the other $\bt_i$ are called \emph{good}. For each good $\bt_i,$ no adjacent cube is bad. Among the $2^k$ many such adjacent cubes $\cC$, there is at least one with the property that
\[
\bt \leq \bt_i
\qquad (\bt \in \cC).
\]

We consider
\begin{align*}
&\sum_{i} \lam(H_n\cap (H_{n'}+\bt_i))\\
=&\sum_{\bt_i \text{ good}} \lam(H_n\cap (H_{n'}+\bt_i))+\sum_{\bt_i \text{ bad}}
\lam(H_n\cap (H_{n'}+\bt_i)).
\end{align*}
The good sum can be treated with the integral trick and thereby bounded from above by $\lambda(A_n)
\lambda(A_{n'})$, regardless of $\bgam.$ 

For the bad sum, suppose $\bt_i$ is bad. Then there is an adjacent cube that intersects the boundary of $[0,1/n]^k.$ Let $\bt'_i$ be its nearest point on the boundary. 
Then $\bt'_i \leq \bt_i$ and
\[
\lambda(H_n\cap (H_{n'} + \bt_i))\leq \lambda(H_n\cap (H_{n'}+ \bt'_i)). 
\]
Observe that distinct $\bt_i'$ are at least $\gcd(n,n')/(nn')$ away from each other. It can happen that different $\bt_i$ give rise to the same $\bt'_i$.
This multiplicity is at most $2^k$, the number of corners of a cube. 

Finally, we bound the number of such $\bt_i'$ 
gotten from bad points such that $H_n$ and 
$H_{n'}+\bt'_i$ 
intersect. 
As $H_n$ is contained in a cube of side length $O(n^{\tau-1-1/k}),$
\begin{align*}
&\# \{i:
\bt'_i \text{ bad}, \quad H_n\cap (H_{n'} + \bt'_i) \neq\emptyset\}
\\ &\ll 
\gcd(n,n')^k
\left(1+
\frac{{n^{\tau-1-1/k}
+ (n')^{\tau-1-1/k}}}{\gcd(n,n')/(nn')} \right)^{k-1}.
\end{align*}
Since
\[
\lam(H_n\cap (H_{n'}+\bt_i')) \leq \min\{ \lam(H_n), \lam(H_{n'}) \},
\]
this completes the proof for general $\bgam$.
\end{proof}

\begin{thm}
\label{thm: quasi-independence part two}
Under the assumptions of Theorem \ref{thm: quasi-independence part one}, if $\bgam_{i,j} \in \bR^k$ for $i,j \in \bN$, where $k \ge 3$, then
\[
\sum_{n,n' \le N}
\lam(A_n\cap (A_{n'}+\bgam_{n,n'}))
\le \left( \sum_{n \le N} \lam(A_n)
\right)^2
+ O \left(
\sum_{n \le N} \lam(A_n)
\right).
\]
\end{thm}

\begin{rem} If $k = 2$, then we get a factor of $\frac{n}{\varphi(n)}$ in the error term, i.e.
\[
\sum_{n,n' \le N}
\lam(A_n\cap (A_{n'} + \bgam_{n,n'}))
\le \left( \sum_{n \le N} \lam(A_n)
\right)^2
+ O \left(
\sum_{n \le N} \frac{n}{\varphi(n)}\lam(A_n)
\right).
\]
If $\psi$ is monotonic, then we can average this factor away using partial summation.
\end{rem}

\begin{proof}
By Theorem \ref{thm: quasi-independence part one}, we have
\begin{align*}
&\sum_{n,n'=1}^N \lam(A_n\cap (A_{n'}+\bgam_{n,n'})) \\
&\le 
2\sum_{n' \le n}
\Biggl(
\lam(A_n)\lam(A_{n'})
\\ 
& \quad +
O \left(
\frac{\lam(A_n)}{n^k}\gcd(n,n')^k
\left(1+\frac{n^{\tau-1-1/k}
+ (n')^{\tau-1-1/k}} {\gcd(n,n')/(nn')}\right)^{k-1}
\right)
\Biggr) \\
&\le S_1 + O(S_2 + S_3),
\end{align*}
where
\begin{align*}
S_1 = \left( \sum_{n \le N} \lam(A_n) \right)^2, \qquad
S_2 = \sum_{n' \le n \le N}
\frac{\lam(A_n)}{n^k} \gcd(n,n')^k,
\end{align*}
and
\[
S_3 = \sum_{n' \le n \le N}
\frac{\lam(A_n)}{n^k} \gcd(n,n')^k
\left(
\frac{(n')^{\tau-1-1/k}}{\gcd(n,n')/(nn')}
\right)^{k-1}.
\]
As $k-1\geq 2$,
\begin{align*}
S_2
& \leq \sum_{n \le N}
\sum_{r \mid n}
\frac{\lam(A_n)}{n^k}
\sum_{\substack{
n'\leq n \\ n' \equiv 0 
\mmod r}} r^k = \sum_{n \le N} \lam(A_n)
\sum_{r \mid n}
\frac{n}{r}
\left(\frac{n}{r} 
\right)^{-k} 
\\
&= \sum_{n \le N} \lam(A_n)
\sum_{s \mid n}s^{-(k-1)} \ll \sum_{n\le N} \lam(A_n).
\end{align*}
As $1-1/k\ge 1/2$,
\begin{align*}
S_3
&\le \sum_{n \le N}
\sum_{r \mid n}
\frac{\lam(A_n)}{n^k}
\sum_{\substack{n'\leq n
\\ n' \equiv 0 \mmod r}} 
r^k 
\frac{(nn'/r)^{k-1} } {(n')^{k-1}
(n')^{(1-k\tau)(1-1/k)}}
\\
&= \sum_{n \le N}
\frac{\lam(A_n)}{n}
\sum_{r \mid n} r
\sum_{\substack{n' \leq n \\ n' \equiv 0 \mmod r}}
\frac{1}{(n')^{(1-k\tau)(1-1/k)}} \\
&\ll \sum_{n \le N} 
\frac{\lam(A_n)}{n}
\sum_{r \mid n} r 
\frac{1}{r^{(1-k\tau)(1-1/k)}}  
(n/r)^{1-(1-k\tau)(1-1/k)} \\
&\ll \sum_{n \le N} \lam(A_n) 
\sum_{r \mid n} 
\frac{1}{n^{(1-k\tau)(1-1/k)}}
\ll \sum_{n \le N} \lam(A_n).
\end{align*}
\end{proof}

\bigskip

Next, we establish functional counterparts to the above theorems. We will replace each set $A_n$ by a smooth approximation $f_n$ supported on $A_n$. Specifically, we will associate to $(A_n)_{n=1}^\infty$ an admissible function system $(f_n)_{n=1}^\infty$, so that $f_n \le 1_{A_n} \le 1$ pointwise.

Recall that $A_n$ is a union of strongly star-shaped sets, each of which is a union of rectangles with disjoint interiors. 
The approximation scheme in \S \ref{sec: rectangle decomposition} dictates that we need to consider $f_n$ being a sum of translated copies of $b_\cR$, for various rectangles $\cR$ that comprise $A_n.$ The bump function $\bb_\cR$ is supported on $\cR$, since $\bb$ is supported on $[-1,1]^k$. 

\begin{thm}
\label{thm: functional quasi-independence part one}
Let $\psi: \bN \to [0,1)$. Consider an admissible function system $(f_n)_{n=1}^\infty$ as above. Then, for some constant $C= C_k>0,$ for all $n,n' \in \bN$ we have
\begin{align*}
\lam(f_n f_{n'}(\cdot +\bgam)) &\le C\lam(f_n)\lam(f_{n'}) \\
& \quad + O \Bigg( \min 
\left \{
\frac{\lam(f_n)}{n^k}, \frac{\lam(f_{n'})}
{(n')^k} \right \} \gcd(n,n')^k \\
& \qquad \cdot \left(1+\frac{n^{\tau-1-1/k} + (n')^{\tau-1-1/k}}{\gcd(n,n')/(nn')}\right)^{k-1} \Bigg),
\end{align*}
uniformly for all $\bgam \in \mathbb{R}^k.$ 
\end{thm}

\begin{proof}
Observe that
$
\lambda(f_n f_{n'}(\cdot + \bgam)) \leq \lambda(A_n\cap (A_{n'} - \bgam)).
$
We may assume that our admissible function system is homogeneous, since Lebesgue measure is translation-invariant.
Now Theorem \ref{thm: quasi-independence part one} gives
\begin{align*}
\lam(f_n f_{n'}(\cdot+\bgam)) &\le \lam(A_n)\lam(A_{n'}) \\
& \quad + O \Bigg( \min 
\left \{
\frac{\lam(A_n)}{n^k}, \frac{\lam(A_{n'})}
{(n')^k} \right \} \gcd(n,n')^k \\
& \qquad \qquad \cdot \left(1 + \frac{n^{\tau-1-1/k} + (n')^{\tau-1-1/k}}{\gcd(n,n')/(nn')}\right)^{k-1} \Bigg).
\end{align*}
We have in general
\[
\lambda(f_n)\leq \lam(A_n).
\]
The other direction holds up to a constant factor:
\[
\lam(A_n) \ll
\lambda(f_n).
\]
We conclude that for some constant $C>0,$
\begin{align*}
\lam(f_n f_{n'}(\cdot+\bgam)) &\le C\lam(f_n)\lam(f_{n'}) \\
& \quad + O \Bigg( \min 
\left \{
\frac{\lam(f_n)}{n^k}, \frac{\lam(f_{n'})}
{(n')^k} \right \} \gcd(n,n')^k \\
& \qquad \qquad \cdot \left(1+
\frac{n^{\tau-1-1/k} + (n')^{\tau-1-1/k}}{\gcd(n,n')/(nn')}
\right)^{k-1} \Bigg).
\end{align*}
\end{proof}

\begin{thm}
\label{thm: functional quasi-independence part two}
Under the assumptions of Theorem \ref{thm: functional quasi-independence part one}, if $\bgam_{i,j} \in \bR^k$ for $i,j \in \bN$, where $k \ge 3$, then
\[
\sum_{n,n' \le N} 
\lam(f_n f_{n'}(\cdot +\bgam_{n,n'}))
\le C\left( \sum_{n \le N} \lam(f_n)
\right)^2
+ O \left(
\sum_{n \le N} \lam(f_n)
\right).
\]
Here $C>0$ is as in Theorem \ref{thm: functional quasi-independence part one}.
\end{thm}

\begin{proof}
The proof is similar to that of Theorem \ref{thm: quasi-independence part two}, using Theorem \ref{thm: functional quasi-independence part one} in place of Theorem \ref{thm: quasi-independence part one}.
\end{proof}

We now make some remarks on the constant $C>0.$ It depends on $k$, as well as on the choice of bump functions $\ome_1, \ome_2$ from \S \ref{admissible}. By choosing these to be sufficiently close to the indicator function of $\{ x: |x| < 1 \}$ and
$\{ x: 1/2 \leq |x| \leq 1\}$, respectively, we may take $C$ to be arbitrarily close to $1.$ We record this observation in the following lemma.

\begin{lem}
\label{lem: divergence constant}
The constant $C$ in Theorems \ref{thm: functional quasi-independence part one} and \ref{thm: functional quasi-independence part two} can be made arbitrarily close to one by choosing the bump functions $\ome_1, \ome_2$ in \S \ref{admissible} accordingly.
\end{lem}

\section{The finale}
\label{finale}

\subsection{On the inhomogeneous generalisation of Gallagher's theorem}

We now establish Theorem \ref{FIG}. The convergence part was proved in \cite[\S 3.5]{CT2024}, since monotonicity was not used there. 

For the divergence part, we construct $(A_n)$ as in Example \ref{MultOK}, and associate a smoothing $(f_n)$ as in \S \ref{admissible}. When $n \in D_m$, this involves $\asymp m^{k-1}$ many disjoint sets $A_n(\bd)$, each of volume $\asymp \psi(n)$. Therefore
\begin{equation}
\label{fnlower}
\lam(f_n) \gg m^{k-1} \psi(2^m) \qquad (n \in D_m).
\end{equation}
By \eqref{SeriesDiverges} and the Cauchy condensation test, this confirms \eqref{DivTech}. The ETP and VTP are immediate for $\lam$, and Theorem \ref{thm: functional quasi-independence part two} gives \eqref{FunctionalLebesgueSecond}. Lemmas \ref{FunctionalDivergence} and \ref{lem: divergence constant} now deliver the desired conclusion. 

\subsection{The flat case}

To obtain the flat convergence result, Theorem \ref{FlatConvergence}, we use the rectangular decomposition $\fC_n$ from \S \ref{sec: rectangle decomposition}, working with $\mu_0$ from \S \ref{FlatETPsection}. We may assume that $\alp_2 \cdots \alp_\ell \ne 0$. Fixing $\varpi \in (\ome(\alp_2, \ldots, \alp_k), \ell)$, we may also assume that
\[
\psi(n) \ge n^{-\ell/\varpi}.
\]
Indeed, we may replace $\psi(n)$ by $\tilde \psi(n) = \psi(n) +  n^{-\ell/\varpi}$, since $\sum_n \tilde \psi(n) < \infty$.

Since $\psi(n) \ge n^{-\ell/ \varpi}$, we may implement our rectangular decomposition with
\begin{equation}
\label{ETPdiop}
\max \{ nd_j: 1 \le j \le \ell \} \gg
n^{-1/\varpi},
\end{equation}
denoting by $f_n$ the resulting sum. We may then apply Theorem \ref{FlatETP} to see that ETP is satisfied. Note in addition that $f_n \gg 1$ on $A_n^\times$, for any sufficiently large $n \in \bN$. 

By Lemma \ref{FunctionalConvergence}, it remains to confirm \eqref{ConvTech}. Observe that
\[
\lam(\bx \mapsto A_n^*(\bd; \bx)) \asymp \psi(n),
\]
where $\lam$ is Lebesgue measure on $[0,1)^k$. As $f_n$ is a sum of $O(\log^{k-1}(1/\psi(n)))$ many of these functions, and $\log(1/\psi(n)) \ll \log n$, we have
\[
\lam(f_n) \ll \psi(n) (\log n)^{k-1}
\]
and hence
\eqref{ConvTech}. This completes the proof of Theorem \ref{FlatConvergence}.

\bigskip

For the flat divergence case, Theorem \ref{FlatDivergence},
we use the admissible set system given in Example \ref{MultOK}. When $n \in D_m$, this involves $\asymp m^{k-1}$ many disjoint sets $A_n(\bd)$, each of volume $\asymp \psi(n)$. We thus have \eqref{fnlower}.
In light of \eqref{SeriesDiverges} and the Cauchy condensation test, this confirms \eqref{DivTech}. Here we have smoothed to obtain an admissible function system $(f_n)$, as described in \S \ref{admissible}. As explained in Example~\ref{MultOK}, we may arrange to have
\eqref{eqn: tau} and hence also \eqref{ETPdiop}. The ETP and VTP for $\mu_0$ follow from
Theorems \ref{FlatETP} and~\ref{FlatVTP}, respectively, while Theorem \ref{thm: functional quasi-independence part two} gives \eqref{FunctionalLebesgueSecond}. Theorem \ref{FlatDivergence} now follows from Lemmas \ref{FunctionalDivergence} and \ref{lem: divergence constant}. 

\subsection{The curved case}

For Theorems \ref{CurvedConvergence} and \ref{CurvedDivergence}, we require a classical result on the Fourier decay of 
smooth measures on curved hypersurfaces \cite[\S 8.3.1 Theorem 1]{Ste1993}.

\begin{thm}
[Decay on curved hypersurfaces]
\label{thm: surface fourier decay} 
Let $\cU$ be an open subset of a smooth hypersurface $\cM$ in $\bR^k$, on which the Gaussian curvature is non-vanishing. Let $\mu$ be a smooth measure supported on $\cU$. Then
\[
|\hat \mu(\bxi)| \ll \| \bxi \|_\infty^{(1-k)/2}.
\]
\end{thm}

To prove Theorem \ref{CurvedConvergence}, we use the rectangular decomposition $\fC_n$ from \S \ref{sec: rectangle decomposition}. Similarly to the proof of Theorem \ref{FlatConvergence}, we may assume that $\psi(n) \ge n^{-2}$, and then argue that
\[
\lam(f_n) \ll \psi(n) (\log n)^{k-1},
\]
furnishing \eqref{ConvTech}.

For the ETP, we use Theorem \ref{CurvedETP}. We may assume that $\mu$ is supported on $[0,1]^k$, since $W_k^\times \subset [0,1)^k$. We also require that
\[
(d_1 \cdots d_k)^{-(1-\sigma/k)} = o(n^k) \qquad (n \to \infty).
\]
According to our decomposition, 
\[
d_1 \cdots d_k\asymp \psi(n)/n^k.
\]
We therefore require that
\[
\psi(n)^{-(1-\sigma/k)}
= o(n^\sigma),
\]
which is equivalent to 
\[
n^{-\sigma/(1-\sigma/k)} = o(\psi(n)).
\]

We are dealing with the case where $\sum_n \psi(n)\log^{k-1} n<\infty.$ We introduce an auxiliary approximation function
\[
\tilde{\psi}(n) =
\max\{ \psi(n),
(n\log^{k+1} n)^{-1} \},
\]
noting that
$W^\times_k(\psi) \subseteq W^\times_k(\tilde{\psi}).$ We now consider the approximation with $\tilde{\psi}.$ Notice that $\tilde{\psi}(n)\geq 1/(n\log^{k+1} n)$ and
\[
\sum_n \tilde{\psi}(n)\log^{k-1} n<\infty.
\]
Furthermore
\[
n^{-\sigma/(1-\sigma/k)} = o(\tilde{\psi}(n)),
\]
as long as $\sigma/(1-\sigma/k)>1.$ This is the case if $\sigma > 1/(1+1/k) = k/(k+1)$. 

Since we can take $\sig = (k-1)/2$, we see that as long as $k\geq 3,$ smooth measures on curved hypersurfaces have a large enough decay parameter $\sigma$. Theorem \ref{CurvedConvergence} now follows from Lemma \ref{FunctionalConvergence}, for all $k\geq 3$. In the case $k=2,$ we may instead invoke \cite{BL2007}.

\bigskip

Finally, to prove Theorem \ref{CurvedDivergence}, we apply Theorems \ref{CurvedETP}, \ref{CurvedVTP} and \ref{thm: functional quasi-independence part two}, together with Lemma \ref{lem: divergence constant}, to the admissible function system described in Example~\ref{MultOK}. Indeed, we see from \eqref{eqn: tau} that ETP follows from Theorem~\ref{CurvedETP}. Moreover, the condition that $k\geq 5$ ensures that we have the required Fourier decay exponent $\sigma\geq 2$ for Theorem \ref{CurvedVTP}.

\end{document}